\documentclass{LMCS}

\usepackage{enumerate}
\usepackage{hyperref}
\usepackage{amssymb,amsmath,amsfonts,amsthm,proof}
\theoremstyle{plain}

\newcommand{\csr}{\ensuretext{\sc (cs)}}

\newcommand{\mdr}{(\dmd)}
\newcommand*{\sor}{\mathcal{S}}
\newcommand*{\var}{\mathrm{Var}}
\newcommand*{\fml}{\ensuretext{\rm Fm}}
\def \ensuretext{\textrm }
\newcommand{\lgc}[1]{\ensuremath{{\sf #1}}}
\newcommand{\der}[1]{\ensuremath{\vdash_{\lgc{#1}}}}
\newcommand{\mdl}[1]{\models_{\lgc{#1}}}
\newcommand*{\mg}{{\mathcal G}}
\newcommand*{\mh}{{\mathcal H}}
\newcommand*{\h}{\mid}
\newcommand*{\mbx}{\Box}
\newcommand*{\dmd}{\Diamond}
\newcommand*{\seq}{\Rightarrow}
\newcommand*{\isdef}{=_{\ensuretext{\tiny def}}}
\newcommand*{\ltoe}{\triangleleft}
\newcommand*{\rtoe}{\not{\!\ltoe}\ }

\newcommand{\De}{\Delta}
\newcommand{\Ga}{\Gamma}
\newcommand{\Si}{\Sigma}
\newcommand{\The}{\Theta}

\newcommand{\lam}{\lambda}

\newcommand*{\mdp}{\ensuretext{\sc (mp)}}
\newcommand{\nec}{(\ensuretext{\sc nec})}
\newcommand{\cutr}{\ensuretext{\sc (cut)}}
\newcommand{\mmr}{(\mbx)}

\newcommand{\olr}{(\lor\!\!\seq)}
\newcommand{\orr}{(\seq\!\!\lor)}
\newcommand{\alr}{(\land\!\!\seq)}
\newcommand{\arr}{(\seq\!\!\land)}
\newcommand{\idr}{\ensuretext{\sc (id)}}
\newcommand{\botr}{(\bot\!\!\seq)}
\newcommand{\topr}{(\seq\!\!\top)}
\newcommand{\wlr}{\ensuretext{\sc (wl)}}

\newcommand{\wrr}{\ensuretext{\sc (wr)}}

\newcommand{\clr}{\ensuretext{\sc (cl)}}

\newcommand{\ilr}{(\to\seq)}
\newcommand{\irr}{(\seq\to)}
\newcommand{\nrr}{(\seq\!\lnot)}
\newcommand{\nlr}{(\lnot\!\seq)}
\newcommand{\comr}{\ensuretext{\sc (com)}}
\newcommand{\ecr}{\ensuretext{\sc (ec)}}
\newcommand{\ewr}{\ensuretext{\sc (ew)}}

\def\doi{7 (2:10) 2011}
\lmcsheading%
{\doi}
{1--27}
{}
{}
{Jul.~25, 2010}
{May\phantom.~17, 2011}
{}

\begin{document}

\title[Towards a Proof Theory of G{\"o}del Modal Logics]{Towards a Proof Theory of G{\"o}del Modal Logics\rsuper*}

\author[G.~Metcalfe]{George Metcalfe\rsuper a}	
\address{{\lsuper a}Mathematics Institute, University of Bern, Sidlerstrasse 5, Bern 3012, Switzerland}	
\email{george.metcalfe@math.unibe.ch}  
\thanks{{\lsuper a}The first author acknowledges support from Swiss National Science Foundation grant 20002{\_}129507}	
\author[N.~Olivetti]{Nicola Olivetti\rsuper b}	
\address{{\lsuper b}LSIS-UMR CNRS 6168, Universit{\'e} Paul C{\'e}zanne, 
Campus de Saint J{\'e}r{\^o}me, Avenue Escadrille Normandie-Niemen, 
13397 Marseille Cedex 20, France}	
\email{nicola.olivetti@univ-cezanne.fr}  
\titlecomment{{\lsuper*}A precursor to this paper, covering only the box fragment of the logics, has appeared as~\cite{MO09}.}


\begin{abstract}
  \noindent 
  Analytic proof calculi are introduced for box and diamond fragments of 
  basic modal fuzzy logics that combine the Kripke semantics 
  of  modal logic K with the many-valued semantics of G{\"o}del logic. The calculi are 
   used to establish completeness and complexity results for these 
  fragments.
  \end{abstract}

\maketitle


\section{Introduction}

A broad spectrum of concepts spanning necessity, knowledge, belief, obligation, and spatio-temporal relations   
have been investigated in the field of modal logic (see, e.g.,~\cite{cha:mod}), while notions relating to truth degrees 
such as vagueness, possibility, and uncertainty have received careful attention in the study of fuzzy logics (see, 
e.g.,~\cite{haj:met,met:book}). Relatively little attention, however, has been paid to logics combining these  
approaches, that is, to \emph{modal fuzzy logics}. Ideally, a systematic development of these logics would provide a 
unified approach to a range of topics considered in the literature such as fuzzy belief~\cite{HEGG94,GHE03}, 
spatial reasoning in the presence of vagueness~\cite{SDK09}, 
fuzzy similarity measures~\cite{GR99}, and fuzzy description logics, which may be understood, analogously to classical 
description logics, as multi-modal fuzzy logics (see, e.g.,~\cite{straccia01a,haj:making,Bobillo09}).

Fuzzy modal logics developed for particular applications are typically  situated quite high up in the spectrum of modal logics, 
e.g., at the level of the logic $\lgc{S5}$ (see, e.g.,~\cite{haj:met}) or based on Zadeh's minimal fuzzy logic 
(see, e.g.,~\cite{zhang06}). On the other hand, general approaches dealing with many-valued modal logics, such as~\cite{Fitting91,Fitting92}, 
have concentrated mainly on the finite-valued case. In particular,  Priest~\cite{priest:many} and Bou et al.~\cite{BEGR08,BEGR09} have provided  
frameworks  for studying many-valued modal logics but their results so far (e.g., for axiomatizations and decidability) 
relate mostly to finite-valued modal logics. The general strategy, followed also in this paper, is to consider logics based 
either on standard Kripke frames or Kripke frames where the accessibility 
 relation between worlds is many-valued (fuzzy). Propositional connectives are interpreted using the given (fuzzy) logic at individual worlds, 
while the values of modal formulas $\mbx A$ and $\dmd A$ are calculated using the infima and suprema, respectively, of values of $A$ 
at accessible (to some degree) worlds. Validity can then be defined as usual as truth (e.g., taking the value 1) at all worlds of all models. 
Let us emphasize, however, that this approach diverges significantly from certain other developments in the literature. In particular, 
intuitionistic and intermediate logics extended with modalities, as investigated in, e.g.,~\cite{Sim94,Wolter97}, make use of two accessibility 
relations for Kripke models, one for the modal operator and another for the intuitionistic connectives. Also, the modalities added to fuzzy 
logics in works such as \cite{haj:very,met:mod} represent  truth stressers such as ``very true'' or ``classically true'' 
and, unlike the modalities considered here, may be interpreted as unary functions on the real unit interval.

In this paper, we narrow our focus on the fuzzy side to propositional G\"odel logic $\lgc{G}$, the infinite-valued version of a family of 
finite-valued logics introduced by G\"odel in~\cite{Goedel32}, axiomatized by Dummett in~\cite{Dummett59} by adding the prelinearity 
schema $(A \to B) \lor (B \to A)$ to an axiomatization of propositional intuitionistic logic. Aside from being an important fuzzy and 
intermediate logic, there are good practical reasons to focus on $\lgc{G}$ in modal contexts. As noted in~\cite{BEGR08}, $\lgc{G}$ is the only 
fuzzy logic whose modal analogues with a fuzzy accessibility relation admit the schema $\mbx(A \to B) \to (\mbx A \to \mbx B)$ 
(roughly speaking, since $\lgc{G}$,  unlike other fuzzy logics, admits both weakening and contraction). Also, a multimodal 
variant of this logic (restricted, however, to finite models) has already been proposed  as the basis for a fuzzy description logic 
in~\cite{Bobillo09}.  Caicedo and Rodr{\'i}guez have already provided axiomatizations for the box and diamond fragments of a G\"odel modal logic 
based on fuzzy Kripke frames in~\cite{CR08}, observing  that the box fragment is also characterized by standard Kripke frames and 
does not have the finite model property, while, conversely, the diamond fragment has the finite model property but is not 
characterized by standard Kripke frames. 

In this work, we introduce (the first) analytic proof systems for fragments of the G\"odel modal logic studied by Caicedo and Rodr{\'i}guez  and 
also for the ``other'' diamond fragment based on standard Kripke frames,  the broader aim being to initiate a general investigation into the 
proof theory of modal fuzzy logics (e.g., as undertaken for propositional and first-order fuzzy logics in~\cite{met:book}). A wide range 
of proof systems have been developed for G\"odel logic, including the sequent calculi of Sonobe~\cite{son:gen} and Dyckhoff~\cite{Dyckhoff99}, 
decomposition systems of Avron and Konikowska~\cite{AvrKon:00}, graph-based methods of Larchey-Wendling~\cite{Lar07}, and 
goal-directed systems of Metcalfe et al.~\cite{met:book}. Here we extend the sequent of relations calculus of Baaz and 
Ferm{\"u}ller~\cite{Baaz:1999:ACP}, which provides a particularly elegant and suitable framework for investigating computational properties such as 
complexity and proof search, and also the hypersequent calculus of Avron~\cite{Avron91b}, which is 
better suited to extensions to the first-order level and other logics. More precisely, we provide sequent of relations calculi for all three fragments and 
a hypersequent calculus for the box fragment. We are then able to use these calculi to obtain PSPACE-completeness results for the fragments 
and constructive completeness proofs for Hilbert-style axiomatizations. In the final section, we discuss connections with related work on 
fuzzy description logics, modal intermediate logics, and other proof frameworks, and consider the problems involved with extending 
this approach to the full logics.


\section{G\"odel Logic}

\subsection{Syntax and Semantics}

\begin{figure}[tbp]
\footnotesize
\[
\begin{array}{rlcrl}
(A1)	& A \to (B \to A)			& \quad\quad & 
(A8)	& (A \to B) \to ((C \to A) \to (C \to B))\\
(A2)  & (A \land B) \to A 			&  &
(A9)	& (A \to (B \to C)) \to (B \to (A \to C))\\
(A3)  & (A \land B) \to B			& &	
(A10) &
((A \to C) \land (B \to C)) \to ((A \lor B) \to C)\\
(A4)  & A \to (B \to (A \land B))	& &	
(A11) & (A \to (B \to C)) \to ((A \land B) \to C)\\
(A5)	& (\bot \to A) \land (A \to \top) & & 
(A12) &((C \to A) \land (C \to B)) \to (C \to (A \land B))\\
(A6)	& A \to (A \lor B) & & 
(A13) & (A \to (A \to B)) \to (A \to B)\\
(A7)	& B \to (A \lor B)			& &	
(A14) &  (A \to B) \lor (B \to A)
\end{array}
\]
\[
\infer[\mdp]{B}{A & A \to B}
\]
\caption{The Hilbert System $\lgc{HG}$}
\label{figgodhilbert}
\end{figure}
We define G\"odel logic $\lgc{G}$  based on a language ${\mathcal L}_\lgc{G}$ consisting of a fixed countably infinite set 
$\var$ of (propositional) variables, denoted $p,q$, binary connectives $\to$, $\land$, $\lor$, and constants $\bot$, $\top$. We  
call  variables and constants {\em atoms}, denoted $a,b$. 
The set of formulas $\fml_{{\mathcal L}_\lgc{G}}$, with arbitrary members denoted $A,B,C,\dots$, is defined 
inductively as usual, and the complexity of a formula $A$, denoted $|A|$, is defined as the number of connectives 
occurring in $A$. We let $\lnot A\isdef A \to \bot$ and $A \leftrightarrow B \isdef (A \to B) \land (B \to A)$. 
We use $\Ga,\Pi,\Si,\De,\The$ to stand for finite multisets of formulas, written $[A_1,\ldots,A_n]$,  letting 
$[]$ denote the empty multiset of formulas and $\Ga \uplus \De$ the multiset sum of $\Ga$ and $\De$. 
We write $\bigvee \Ga$ and $\bigwedge \Ga$ with $\bigvee [] \isdef \bot$ and $\bigwedge [] \isdef \top$ for iterated disjunctions 
and conjunctions of formulas, and define $\Ga^0 \isdef []$ and $\Ga^{n+1} \isdef \Ga \uplus \Ga^n$ for 
$n \in \mathbb{N}$.

The standard semantics of G\"odel logic is characterized by the G\"odel t-norm $\min$ and its residuum 
$\to_\lgc{G}$, defined on the real unit interval $[0,1]$ by 
\[
x \to_\lgc{G} y = \begin{cases} y & {\rm if } \ x > y\\1 & {\rm otherwise.}\end{cases}
\]
More precisely, a \emph{$\lgc{G}$-valuation} is a function $v: \fml_{{\mathcal L}_\lgc{G}} \to [0,1]$ satisfying
\[
\begin{array}[t]{rcl}
v(\bot) 	& = & 0\\ 
v(\top) 	& = & 1\\
v(A \to B)  & = & v(A) \to_\lgc{G} v(B)\\
v(A \land B) & = & \min(v(A),v(B))\\
v(A \lor B) & = & \max(v(A),v(B)).\\
\end{array}
\]
A formula $A$ is \emph{$\lgc{G}$-valid}, written $\mdl{\lgc{G}} A$, iff $v(A) = 1$ for all $\lgc{G}$-valuations $v$. 

Sometimes it will be helpful to consider extensions of the language ${\mathcal L}_\lgc{G}$ with (finitely many)  
constants, denoted $c,d$, where $\lgc{G}$-valuations satisfy  $v(c) = r_c$ for each 
additional constant $c$ for some fixed $r_c \in [0,1]$.\footnote{Note that the real number $r_c$ does not appear 
in the language, rather it is fixed at the same time as the constant is added to the language.} 
In such cases, $\top$ and $\bot$ are  
considered together alongside the extra constants, with $v(r_\top) =1$ and $v(r_\bot) = 0$. 

A standard axiomatization $\lgc{HG}$ for G\"odel logic in the language ${\mathcal L}_\lgc{G}$ 
is provided in Figure~\ref{figgodhilbert}. $\lgc{HG}$ is complete with respect to the standard semantics, 
and also with respect to both G\"odel algebras, defined as Heyting algebras obeying the prelinearity law 
$\top = (x \to y) \lor (y \to x)$, and linearly ordered (intuitionistic) Kripke models. Analytic proof systems for G\"odel logic, 
where structures appearing in a derivation are constructed from subformulas of the formula to be proved, 
 have been defined in a number of different frameworks. Below we consider two of the most useful of these frameworks, 
 sequents of relations and hypersequents.


\subsection{Sequents of Relations} 

Sequents of relations, consisting of sets of pairs of formulas ordered by the relations $\le$ and $<$, 
were introduced by Baaz and Ferm{\"u}ller in~\cite{Baaz:1999:ACP} as a proof-theoretic 
framework for  G\"odel and other ``projective'' logics. 
 More formally, a \emph{sequent of relations} $\sor$ for a language ${\mathcal L}$ is a finite (possibly empty) 
 set of ordered triples, written
\[
A_1 \ltoe_1 B_1 \h \ldots \h A_n \ltoe_n B_n
\]
where $A_i,B_i \in \fml_{\mathcal L}$ and  $\ltoe_i \in \{<,\le\}$ for $i = 1\ldots n$. A sequent of relations is 
called {\em atomic} if it contains only atoms.

For G\"odel logic, the syntactic symbol $\ltoe \in \{\le,<\}$ is interpreted as the corresponding relation 
over $[0,1]$. That is, for a sequent of relations $\sor$ for ${\mathcal L}_\lgc{G}$,
\begin{quote}
$\mdl{G} \sor$ \quad iff \quad for all $\lgc{G}$-valuations $v$, $v(A) \ltoe v(B)$ for some $(A \ltoe B) \in \sor$,
\end{quote}
and we say in this case that $\sor$ is \emph{$\lgc{G}$-valid}.

Alternatively, we can define the following formula interpretation:
\[
I_S(\{A_i < B_i\}_{i=1}^n \h \{C_j \le D_j\}_{j=1}^m) = \bigwedge_{i=1}^n (B_i \to A_i) \to \bigvee_{j=1}^m (C_j \to D_j).
\]
It follows easily using the deduction theorem for G\"odel logic that $\mdl{G} \sor$ iff  $\mdl{G} I_S(\sor)$.

\begin{figure}[tbp] 
\flushleft
\begin{scriptsize}
Axioms and Structural Rules\\
\[
\begin{array}{ccc}
\infer[{\idr}]{\sor \h A \le A}{} &  & 
\infer[\csr]{\sor}{\sor \h \top \le \bot \h \bot < \bot}\\[.1in]
\infer[\wlr]{\sor \h A \le B}{\sor \h A \le B \h \top \le B} & &
\infer[\wrr]{\sor \h A \le B}{\sor \h A \le B \h  A \le \bot}\\[.1in]
\infer[{\ewr}]{\sor \h A \ltoe B}{\sor} & & \infer[{\comr}]{\sor \h A \ltoe_1 B \h C \ltoe_2 D}{
 \sor \h A \ltoe_1 B \h C \ltoe_2 D \h A \le D &
 \sor \h A \ltoe_1 B \h C \ltoe_2 D \h C \ltoe_1 B}  \\[.1in]
\end{array}
\]
Logical Rules\\
\[
\begin{array}[t]{ccc}
\infer[(\land\ltoe)]{\sor \h A \land B \ltoe C}{
\sor \h A \ltoe C \h B \ltoe C} & &
\infer[(\ltoe\land)]{\sor \h C \ltoe A \land B}{\sor \h C \ltoe A & \sor \h C \ltoe B}\\[.1in]
\infer[(\lor\ltoe)]{\sor \h A \lor B \ltoe C}{\sor \h A \ltoe C & \sor \h B \ltoe C} & & 
\infer[(\ltoe\lor)]{\sor \h C \ltoe A \lor B}{\sor \h C \ltoe A \h C \ltoe B}\\[.1in]
\infer[(\to<)]{\sor \h A \to B < C}{\sor \h B < A & \sor \h B < C} & & 
\infer[(<\to)]{\sor \h C < A \to B}{\sor \h A \le B \h C < B & \sor \h C < \top}\\[.1in]
\infer[(\to\le)]{\sor \h A \to B \le C}{\sor \h \top \le C \h B < A & \sor \h B \le C} & & 
\infer[(\le\to)]{\sor \h C \le A \to B}{\sor \h A \le B \h C \le B} 
\end{array}
\]
\end{scriptsize}
\caption{The Sequent of Relations Calculus $\lgc{SG}$}
\label{figgodsor}
\end{figure}

The sequent of relations calculus $\lgc{SG}$ for G\"odel logic displayed in Figure~\ref{figgodsor} 
consists of logical rules taken from~\cite{Baaz:1999:ACP} together with additional axioms and structural rules 
based on similar calculi for $\lgc{G}$ presented in~\cite{met:book}. The rule $\comr$ (the only rule with more than one non-context 
relation in the conclusion) reflects the linearity of the truth values in G\"odel logic, while the rule $\ewr$ is not strictly necessary 
for completeness, but is useful for constructing derivations. 

Note also that the following helpful axioms and 
rules for the  constants and negation are derivable:
{\footnotesize
\[
\begin{array}{c}
\infer[(<)]{\sor \h \bot< \top}{}  \qquad
\infer[(\bot\!\le)]{\sor \h \bot \le A}{} \qquad 
\infer[(\le\!\top)]{\sor \h A \le \top}{}\\[.1in]
\begin{array}[t]{ccc}
\infer[(\lnot\!<)]{\sor \h \lnot A < C}{\sor \h \bot < A & \sor \h \bot < C} & & 
\infer[(<\!\lnot)]{\sor \h C < \lnot A}{\sor \h A \le \bot & \sor \h C < \top}\\[.1in]
\infer[(\lnot\!\le)]{\sor \h \lnot A \le B}{\sor \h \top \le B \h \bot < A} & & 
\infer[(\le\!\lnot)]{\sor \h B \le \lnot A}{\sor \h A \le \bot \h B \le \bot} 
\end{array}
\end{array}
\]}
Given some fixed notion of $\lgc{L}$-validity, we say that a rule is {\em $\lgc{L}$-sound} if whenever the premises 
are $\lgc{L}$-valid, then so is the conclusion. Conversely, a rule is {\em $\lgc{L}$-invertible} if whenever the conclusion 
is $\lgc{L}$-valid, then so are the premises. The key observation for $\lgc{SG}$ is that, unlike several other calculi 
for G\"odel logic, the logical rules are not only $\lgc{G}$-sound, but also $\lgc{G}$-invertible~\cite{Baaz:1999:ACP}. 
Since upwards applications of the logical rules terminate (a standard argument), this means that the question of the 
$\lgc{G}$-validity of a sequent of relations can be reduced to the question of the $\lgc{G}$-validity of atomic 
sequents of relations.  

Moreover, the following lemma provides a perspicuous and useful characterization of $\lgc{G}$-valid atomic sequents of relations.

\begin{lem}
An atomic sequent of relations $\sor$ (for ${\mathcal L}_\lgc{G}$ possibly with additional constants) is 
$\lgc{G}$-valid iff there exists $(a_i \ltoe_i a_{i+1}) \in \sor$ for $i=1\ldots n$ such that 
one of the following holds:

\begin{enumerate}[\em(1)]

\item   $a_1 = a_{n+1}$ and $\ltoe_i$ is $\le$ for some $i \in \{1,\ldots,n\}$
                
\item   $a_1 = \bot$ and $\ltoe_i$ is $\le$ for some $i \in \{1,\ldots,n\}$

\item   $a_{n+1} = \top$ and $\ltoe_i$ is $\le$ for some $i \in \{1,\ldots,n\}$

\item   $a_1 = c$, $a_{n+1} = d$, and $r_c < r_d$

\item   $a_1 = c$, $a_{n+1} = d$, $r_c = r_d$, and $\ltoe_i$ is $\le$ for some $i \in \{1,\ldots,n\}$.

\end{enumerate}
\end{lem}
\proof
It is clear that $\sor$ is $\lgc{G}$-valid if any of the above conditions are met. For the other direction, we proceed by 
induction on the number of different variables occurring in $\sor$. Note first that if one of
$a  \le a$, $a \le \top$, $\bot \le a$, $c < d$ for $r_c < r_d$, or $c \le d$ for $r_c = r_d$ occurs in $\sor$, then we are done. 
This takes care of the base case. For the inductive step, we fix a  variable $q$ occurring in $\sor$, and define:
\[
\begin{array}[t]{lcl}
\sor_< & \isdef & \{a < b \mid \{a < q, q < b\} \subseteq \sor\}\\[.05in]
\sor_\le & \isdef & \{a \le b \mid \{a \ltoe_1 q, q \ltoe_2 b\} \subseteq \sor,  \le \in \{\ltoe_1, \ltoe_2\}\}\\[.05in]
\sor' &  \isdef & \{a \ltoe b \in \sor \mid a \neq q, b \neq q\} \cup \sor_< \cup \sor_\le.
\end{array}
\]
$\sor'$ has fewer different variables than $\sor$. So if $\sor'$ is $\lgc{G}$-valid, then
applying the induction hypothesis to $\sor'$, we have $(a_i \ltoe_i a_{i+1})
\in \sor'$ for $i=1\ldots n$, satisfying one of the above conditions. But then by replacing
the inequalities $a_i \ltoe_i a_{i+1}$ that occur in $\sor_<$ or $\sor_\le$ appropriately by $a_i \ltoe' q$ and 
$q \ltoe'' a_{i+1}$, we get that one of the conditions holds for $\sor$.
Hence it is sufficient to show that $\sor'$ is $\lgc{G}$-valid. Suppose otherwise,
i.e.,   that there exists a $\lgc{G}$-valuation
$v$ such that $v(a) \ltoe v(b)$ does not hold for any $a \ltoe b \in \sor'$.
We show for a contradiction that $\sor$ is not $\lgc{G}$-valid. Let
\[
x = \min\{v(a) \mid a \ltoe q \in \sor\} \quad \textrm{ and } \quad y = \max\{v(b) \mid q \ltoe b \in \sor\}.
\]
Note first that $x \ge y$. Otherwise it follows that for some $a$, $b$, we have $\{a \ltoe_1 q, q \ltoe_2 b\} \subseteq \sor$ and 
$v(a) < v(b)$. But then $(a \ltoe b) \in \sor'$ so $v(a) \ge v(b)$, a contradiction. Hence there are  two cases. If $x > y$, then we extend 
$v$ such that $x > v(q) > y$. For any $(a \ltoe q)$ or $(q \ltoe b)$ in $\sor$, we have $v(a) \ge x > v(q) > y \ge v(b)$. So 
$\sor$ is not $\lgc{G}$-valid, a contradiction. Now suppose that $x=y$ and extend $v$ such that $v(q)=x$. We must have atoms $a_0$, $b_0$ 
such that  $(a_0 < q)$ and $(q < b_0)$ are in $\sor$ and $v(a_0) = v(b_0) = v(q)$. Now consider any $(a \ltoe_1 q)$ or $(q \ltoe_2 b)$ 
in $\sor$. Since $(a \ltoe_1 b_0)$ and $(a_0 \ltoe_2 b)$  are in $\sor'$, $v(a) \ltoe_1 v(q) = v(b_0)$ and $v(a_0) = v(q) \ltoe_2 v(b)$ 
cannot hold. So $\sor$ is again not $\lgc{G}$-valid, a contradiction. \qed

Let us call an atomic sequent of relations $\sor$ \emph{saturated} if whenever $\sor$ occurs as the conclusion of a logical rule, 
$\comr$, $\csr$, $\wlr$, or $\wrr$, then $\sor$ also occurs as one of the premises. In other words, the sequent of relations 
is ``closed'' under applications of these rules. Then saturated $\lgc{G}$-valid atomic sequents of relations have the 
following property:

\begin{lem} \label{lem:saturatedvalid}
Every saturated $\lgc{G}$-valid atomic sequent of relations $\sor$ (for ${\mathcal L}_\lgc{G}$ possibly with additional constants)  
contains either $(a \le a)$ or $(c \ltoe d)$ where  
$r_c \ltoe r_d$.
\end{lem}
\proof
By the previous lemma, each  saturated $\lgc{G}$-valid atomic sequent of relations $\sor$ contains 
$(a_i \ltoe_i a_{i+1})$ for $i=1\ldots n$ such that one of the conditions  (1)-(5)  holds. The claim then 
follows by a simple induction on $n$, where the base case makes use of $\wlr$ and $\wrr$, and the 
inductive step makes use of $\comr$. \qed

\noindent
Note that in the case of sequents of relations  for ${\mathcal L}_\lgc{G}$ 
without extra constants, every saturated $\lgc{G}$-valid atomic sequent of relations, 
and hence every $\lgc{G}$-valid atomic sequent of relations, is derivable in $\lgc{SG}$. Moreover, 
since, as mentioned above, the logical rules are $\lgc{G}$-sound and $\lgc{G}$-invertible, and reduce 
$\lgc{G}$-valid sequents of relations to $\lgc{G}$-valid atomic sequents of relations, 
it follows that:

\begin{thm}
For any sequent of relations $\sor$ for ${\mathcal L}_\lgc{G}$: \ 
$\der{SG} \sor$ iff $\mdl{G} \sor$ iff $\mdl{G} I_S(\sor)$. \qed
\end{thm}


\subsection{A Hypersequent Calculus} 

Hypersequents were introduced by Avron in~\cite{Avron87} as a generalization of Gentzen sequents that allow 
disjunctive or parallel forms of reasoning. Instead of a single sequent, there is a collection 
of sequents that can be ``worked on''  simultaneously. More precisely,  a \emph{(single-conclusion) sequent} $S$ for a 
language ${\mathcal L}$ is (defined here as) an ordered pair consisting of a finite multiset  $\Ga$ of ${\mathcal L}$-formulas 
and a multiset $\De$ containing at most one $\mathcal{L}$-formula, written $\Ga \seq \De$. A 
\emph{(single-conclusion) hypersequent} $\mg$ for ${\mathcal L}$ is a finite (possibly empty) 
multiset of sequents for ${\mathcal L}$, written
 $\Ga_1 \seq \De_1 \h \ldots \h \Ga_n \seq \De_n$ or sometimes, for short, as $[\Ga_i \seq \De_i]_{i=1}^n$. 
 
 We interpret sequents and hypersequents  for G\"odel logic as follows (recalling that 
 $\bigwedge [] \isdef \top$ and $\bigvee [] \isdef \bot$):
\[
\begin{array}{rcl}
I_H(\Ga \seq \De) & \isdef & \bigwedge \Ga \to \bigvee \De\\[.05in]
I_H(S_1 \h \ldots \h S_n) & \isdef  & I_H(S_1) \lor \ldots \lor I_H(S_n).
\end{array}
\] 
Hypersequent calculi admitting cut-elimination have been defined for a wide range of fuzzy logics 
(for details see~\cite{met:book}). In particular, the first example of such a system (modulo a few inessential 
changes) was the calculus $\lgc{GG}$ defined for G\"odel logic by 
Avron in~\cite{Avron91b}. This calculus, displayed in Figure~\ref{figgod}, can be viewed as 
a direct extension of a sequent calculus for intuitionistic logic. Namely, the  
axioms, weakening, contraction, cut, and logical rules, are obtained from standard sequent rules simply by 
adding a hypersequent context $\mg$ to the premises and conclusion. The external weakening and 
contraction rules, $\ewr$ and $\ecr$, reflect the interpretation of ``$\h$'' as a meta-level disjunction, 
while the communication rule $\comr$, corresponding to the prelinearity axiom schema 
$(A \to B) \lor (B \to A)$, is the crucial ingredient in extending the system beyond intuitionistic logic.

\begin{exa}
Consider the following derivation in $\lgc{GG}$:
{\small\[
\infer[{\ecr}]{\seq (p \to q) \lor (q \to p)}{
 \infer[{\orr_1}]{\seq (p \to q) \lor (q \to p) \h \seq (p \to q) \lor (q \to p)}{
 \infer[{\orr_2}]{\seq p \to q \h \seq (p \to q) \lor (q \to p)}{
    \infer[{\irr}]{\seq p \to q \h \seq q \to p}{
     \infer[{\irr}]{p \seq q \h \seq q \to p}{
      \infer[{\comr}]{p \seq q \h q \seq p}{
       \infer[{\idr}]{p \seq p}{} &
       \infer[{\idr}]{q \seq q}{}}}}}}}
\]}
Notice that the hypersequent $(p \seq q \h q \seq p)$ two lines down might be read as just a ``hypersequent translation'' of
the prelinearity axiom $(p \to q) \lor (q \to p)$.
\end{exa}

\begin{thm}[\cite{Avron91b}]\hfill
\begin{enumerate}[\em(1)]
\item For any hypersequent $\mg$ for ${\mathcal L}_\lgc{G}$: \ $\mdl{G} I_H(\mg)$ iff $\der{GG} \mg$.
\item $\lgc{GG}$ admits cut-elimination. \qed
\end{enumerate}
\end{thm}

\noindent
For future reference, we observe that the following rules for the 
defined negation  $\lnot A \isdef A \to \bot$ are derivable in $\lgc{GG}$:
{\small
\[
\infer[\nlr]{\mg \h \Ga, \lnot A \seq}{\mg \h \Ga \seq A} \qquad\qquad
\infer[\nrr]{\mg \h \Ga \seq \lnot A}{\mg \h \Ga, A \seq}
\]}

\begin{figure}[tbp]
\flushleft
\begin{scriptsize}
Axioms and Structural Rules
\[
\begin{array}{ccccc}
\infer[{\idr}]{\mg \h A \seq A}{} & \quad\quad & 
\infer[\botr]{\mg \h \Ga, \bot \seq \Delta}{} & \quad\quad & 
\infer[\topr]{\mg \h \Ga \seq \top}{}\\[.1in]
\infer[{\ecr}]{\mg \h \mh}{\mg \h \mh \h \mh} & &
\infer[{\ewr}]{\mg \h \mh}{\mg} & & 
\infer[{\comr}]{\mg \h \Ga_1, \Ga_2 \seq \De_1 \h \Pi_1, \Pi_2 \seq \De_2}{
\mg \h \Ga_1,\Pi_1 \seq \De_1 & \mg \h \Ga_2,\Pi_2 \seq \De_2}\\[.1in]
\infer[\wlr]{\mg \h \Ga, A \seq \De}{\mg \h \Ga \seq \De} & & 
\infer[\wrr]{\mg \h \Ga \seq A}{\mg \h \Ga \seq } & &
\infer[\clr]{\mg \h \Ga, A \seq \De}{\mg \h \Ga, A, A \seq \De} 
\end{array}
\]
Logical Rules
\[
\begin{array}{ccc}
\infer[\ilr]{\mg \h \Ga, A \to B \seq \De}{\mg \h \Ga \seq A & \mg \h \Ga, B \seq \De} & &
\infer[\irr]{\mg \h \Ga \seq A \to B}{\mg \h \Ga,A \seq B}\\[.1in]
\infer[\alr_1]{\mg \h \Ga, A \land B \seq \De}{\mg \h \Ga,A \seq \De} & &
\infer[\alr_2]{\mg \h \Ga, A \land B \seq \De}{\mg \h \Ga,B \seq \De}\\[.1in]
\infer[\olr]{\mg \h \Ga, A \lor B \seq \De}{\mg \h \Ga, A \seq \De & \mg \h \Ga, B \seq \De} & & 
\infer[\arr]{\mg \h \Ga \seq A \land B}{\mg \h \Ga \seq A & \mg \h \Ga \seq B}\\[.1in]
\infer[\orr_1]{\mg \h \Ga \seq A \lor B}{\mg \h \Ga \seq A} & & 
\infer[\orr_2]{\mg \h \Ga \seq A \lor B}{\mg \h \Ga \seq B}
\end{array}
\]
Cut Rule
\[\infer[\cutr]{\mg \h \Ga_1, \Ga_2 \seq \De}{
\mg \h \Ga_1,A \seq \De & \mg \h \Ga_2 \seq A}
\]
\end{scriptsize}
\caption{The Gentzen System $\lgc{GG}$}
\label{figgod}
\end{figure}

While the logical rules of the sequent of relations calculus $\lgc{SG}$ are invertible, with consequent 
advantages for establishing complexity and interpolation results and building efficient proof systems,  
the hypersequent calculus $\lgc{GG}$ does not have this property. The virtue of the system lies rather 
with its close connection to sequent calculi for intuitionistic logic and the existence of relatively 
straightforward cut-elimination proofs. In particular, this means that $\lgc{GG}$, unlike $\lgc{SG}$, can be easily 
extended to the first-order level or with propositional quantifiers, and used to prove, for 
example, completeness, Herbrand theorems, and Skolemization results (see, e.g.,~\cite{BaaZac00,met:book}).


\section{Adding Modalities}

\subsection{Syntax and Semantics} 

We extend our language ${\mathcal L}_\lgc{G}$ to a modal language ${\mathcal L}_{\mbx\dmd}$ by 
adding the unary operators $\mbx$ and $\dmd$, obtaining a set of formulas $\fml_{{\mathcal L}_{\mbx\dmd}}$. 
For a finite multiset $\Ga = [A_1,\ldots,A_n]$ of ${\mathcal L}_{\mbx\dmd}$-formulas and $\star \in \{\mbx,\dmd\}$, 
we let $\star \Ga \isdef [\star A_1, \ldots, \star A_n]$. G\"odel modal logics are then defined,  following similar ideas 
proposed in \cite{Fitting91,Fitting92,haj:met,CR08,BEGR08}, as generalizations of the modal 
logic $\lgc{K}$ where connectives behave locally at individual worlds as in G\"odel logic. In particular, 
$\lgc{GK}$ and $\lgc{GK^F}$ are G\"odel modal logics based on, respectively,  
standard Kripke frames and Kripke frames with fuzzy accessibility relations.

A \emph{fuzzy Kripke frame} is a pair $F = \langle W, R \rangle$ where 
$W$ is a non-empty set of {\em worlds} and $R: W \times W \to [0,1]$  is a binary 
\emph{fuzzy accessibility relation} on $W$. If $Rxy \in \{0,1\}$ for all $x,y \in W$, then 
$R$ is called {\em crisp} and $F$ is called simply  a  \emph{(standard) Kripke frame}. In this 
case, we often write $R \subseteq W^2$ and $Rxy$ or $(x,y) \in R$ to mean $Rxy = 1$. 
A \emph{Kripke model for $\lgc{GK^F}$} is then a 3-tuple $K = \langle W, R, V \rangle$ where 
$\langle W,R \rangle$ is a fuzzy Kripke 
frame and $V : \var \times W \to [0,1]$ is a mapping, called a \emph{valuation},  extended 
to $V : \fml_{{\mathcal L}_{\mbx\dmd}} \times W \to [0,1]$ as follows:
\[
\begin{array}{rcl}
V(\bot,x) 	    & = & 0\\
V(\top,x)	    & = & 1\\
V(A \to B,x) & = & V(A,x) \to_\lgc{G} V(B,x)\\
V(A \land B,x) & = & \min(V(A,x),V(B,x))\\
V(A \lor B,x)    & = & \max(V(A,x),V(B,x))\\
V(\mbx A,x) 	& = & \inf \{Rxy \to_\lgc{G} V(A,y) \mid y \in W\} \\ 
V(\dmd A,x) 	& = & \sup \{\min(V(A,y), Rxy) \mid y \in W\}.
\end{array}
\]
A {\em Kripke model for $\lgc{GK}$} satisfies the extra condition 
that $\langle W,R \rangle$ is a standard Kripke frame. In this 
case, the conditions for $\mbx$ and $\dmd$  may also be read as:
\[
\begin{array}{rcl}
V(\mbx A,x)  & = & \inf  (\{1\} \cup \{V(A,y) \mid Rxy\})\\
V(\dmd A,x)   & = & \sup  (\{0\} \cup \{V(A,y) \mid Rxy\}).
\end{array}
\]
$A \in \fml_{\mbx\dmd}$ is {\em valid} in  $\langle W,R,V \rangle$ if 
$V(A,x) = 1$ for all $x \in W$. $A$ is \emph{$\lgc{L}$-valid} for $\lgc{L} \in \{\lgc{GK},\lgc{GK^F}\}$, written $\mdl{L} A$, 
 if $A$ is valid in all Kripke models $\langle W,R,V \rangle$ for $\lgc{L}$.
 
An important feature of G\"odel logic and, more particularly, Kripke models for $\lgc{GK^F}$, is that only the 
{\em order} of truth values matters. The following lemma is proved by a straightforward induction on formula 
complexity.

\begin{lem} \label{lemscaletwo}
Let $K = \langle W, R, V \rangle$ be a Kripke model for  $\lgc{GK^F}$ and $h : [0,1] \to [0,1]$ an 
order-automorphism of the real unit interval. Define $K' = \langle W, R', V' \rangle$ where $R'xy = h(Rxy)$ for all $x,y \in W$ and 
$V'(p,x) = h(V(p,x))$ for each $p \in \var$ and $x \in W$. 
Then $V'(A,x) = h(V(A,x))$ for all $A \in \fml_{\mbx\dmd}$ and $x \in W$. \qed
\end{lem}

Moreover, for Kripke models for  $\lgc{GK}$ (which, recall, have a crisp accessibility relation), we can 
shift the values of all formulas above a certain threshold to $1$, while preserving the values of formulas 
below that threshold.

\begin{lem} \label{lemscaleone}
Let $K = \langle W, R, V \rangle$ be a Kripke model for $\lgc{GK}$ and $\lam \in (0,1]$.  Define 
 $K' = \langle W, R, V' \rangle$ where $V'(p, x) = \lam \to_\lgc{G} V(p,x)$ for each $pÊ\in \var$ and  $x \in W$. 
 Then $V'(A, x) = \lam \to_\lgc{G} V(A,x)$  for all  $A \in \fml_{\mbx\dmd}$ and $x\in W$.
\end{lem}
\begin{proof}
We proceed by induction on $|A|$. The base cases are immediate. Suppose 
that $A$ is $B \to C$. Then, using the induction hypothesis for the second step:
\[
\begin{array}{rcl}
V'(B \to C,x) 	& = & V'(B,x) \to_\lgc{G} V'(C,x)\\
			& = & (\lam \to_\lgc{G} V(B,x)) \to_\lgc{G} (\lam \to_\lgc{G} V(C,x))\\
			& = & \lam \to_\lgc{G} (V(B,x) \to_\lgc{G} V(C,x))\\			
			& = & \lam \to_\lgc{G}  V(B \to C,x).\\
\end{array}
\]
If $A$ is $\dmd B$, then, again using the induction hypothesis for the second step:
\[
\begin{array}{rcl}
V'(\dmd B,x) 	& = & \sup  (\{0\} \cup \{V'(B,y) \mid Rxy\})\\	
			& = & \sup  (\{0\} \cup \{\lam  \to_\lgc{G} V(B,y) \mid Rxy\})\\			
			& = & \lam  \to_\lgc{G}  \sup  (\{0\} \cup \{V(B,y) \mid Rxy\})\\
			& = & \lam \to_\lgc{G} V(\dmd B,x).
\end{array}
\]
Cases for the other connectives are very similar.
\end{proof}


\subsection{Box and Diamond Fragments} \label{ss:boxdiamond}

Due to the difficulty of dealing with the full language (see Section~\ref{finalsection}), we focus in this work on 
the ``box'' and ``diamond'' fragments of $\lgc{GK}$ and $\lgc{GK^F}$ 
based on restrictions of the language ${\mathcal L_{\mbx\dmd}}$ to the single modality 
sublanguages ${\mathcal L_\mbx}$ and ${\mathcal L_\dmd}$, respectively. These fragments are 
worthy of investigation since they already contain enough extra expressive power to 
deal with certain modal fuzzy notions. Indeed the general approach of~\cite{BEGR08,BEGR09} begins with a 
treatment of just the addition of the box operator, reflecting the fact that in (classical) modal logics, typically just 
one of the dual modalities is considered primitive.

We use proof-theoretic methods to establish decidability and complexity 
results for these fragments, and  completeness for the axiomatizations:

\begin{enumerate}[$\bullet$]

\item	$\lgc{HGK_\mbx}$ is $\lgc{HG}$ extended with
		\[\small
		\begin{array}{ll}
		(K_\mbx) & \mbx (A \to B) \to (\mbx A \to \mbx B)\\
		(Z_\mbx) & \lnot \lnot \mbx A \to \mbx \lnot \lnot A
		\end{array}
		\]
		\[\small
		\infer[\nec_\mbx]{\mbx A}{A}
		\]
		
\item	$\lgc{HGK_\dmd}$ is $\lgc{HG}$ extended with
		\[\small
		\begin{array}{ll}
		(K_\dmd) & \dmd (A \lor B) \to (\dmd A \lor \dmd B)\\
		(Z_\dmd)  & \dmd \lnot \lnot A \to  \lnot \lnot \dmd A\\
		(F_\dmd)  &	\lnot \dmd \bot
		\end{array}
		\]
		\[\small
		\infer[\nec_\dmd]{(\dmd A \to \dmd B) \lor \dmd C}{(A \to B) \lor C}
		\]
		
\item	$\lgc{HGK^F_\dmd}$ is $\lgc{HGK_\dmd}$ with $\nec_\dmd$ replaced by
		\[\small
		\infer[\nec^*_\dmd]{\dmd A \to \dmd B}{A \to B}
		\]

\end{enumerate}

It is proved in~\cite{CR08} that $\lgc{HGK_\mbx}$ is  complete with 
respect to the box fragments of both $\lgc{GK}$ and $\lgc{GK^F}$ (in other words, these fragments 
coincide and there is no need to define an alternative $\lgc{HGK^F_\mbx}$), and that $\lgc{HGK^F_\dmd}$ 
is complete with respect to the diamond fragment of 
$\lgc{GK^F}$. These results, plus completeness for $\lgc{HGK_\dmd}$ with respect to  the diamond 
fragment of $\lgc{GK}$, 
will follow from our proof-theoretic investigations, as will decidability and complexity results 
for all three fragments. First, however, we note the following interesting feature of 
G\"odel modal logics.

\begin{thm}
The box and diamond fragments of $\lgc{GK}$  do not have the finite model property.
\end{thm}
\proof
Following~\cite{CR08}, consider the  ${\mathcal L_\mbx}$-formula
\[
A = \mbx \lnot \lnot p \to \lnot \lnot \mbx p.
\]
$A$ is valid in all Kripke models for $\lgc{GK}$ with a finite number of worlds, but 
not in the Kripke model 
 $\langle \mathbb{N},R,V \rangle$ where $Rxy$ holds for all 
$x,y \in \mathbb{N}$ and $V(p,x) = 1/(x+1)$ for all $x \in \mathbb{N}$.

Now  consider the ${\mathcal L_\dmd}$-formula:
\[
B = (\dmd p \to \dmd q) \to ((\dmd q \to \bot) \lor \dmd (p \to q)).
\]
We claim first that $B$ is valid in all Kripke models for $\lgc{GK}$ with a finite number of worlds. 
Consider a world $x$. 
If no worlds are accessible to $x$, then $V(\dmd q,x) = 0$ and $V(B,x) = 1$. Otherwise, 
since the model is finite, we can choose an accessible world $y$ such that $V(\dmd q,x) = V(q,y)$. 
If $V(p,y) > V(q,y)$, then $V(\dmd p,x) > V(\dmd q, x)$, so also 
$V(B,x) = 1$. If $V(p,y) \le V(q,y)$, then
 $V(p \to q,y)=1$, so also $V(B,x) = 1$.

On the other hand, consider a Kripke model  $\langle \mathbb{N},R,V \rangle$ for 
$\lgc{GK}$ where $Rxy$ holds iff $x=0$ and $y \ge 1$, $V(p,x) = \frac{1}{2}$ and  $V(q,x) = \frac{1}{2} - \frac{1}{x+2}$ 
for all $x \in \mathbb{N}$. Then $V(\dmd q,0) = V(\dmd p,0) = V(\dmd (p \to q), 0) = \frac{1}{2}$, 
so $V(B,0) = \frac{1}{2}$. \qed

This contrasts with the following result of Caicedo and Rodr{\'i}guez.

\begin{thm}[\cite{CR08}]
 The diamond fragment of $\lgc{GK^F}$ has the finite model property. \qed
\end{thm}


\subsection{Sequents of Relations} 

We extend the definition of validity for sequents of relations from $\lgc{G}$ to 
$\lgc{L} \in \{\lgc{GK},\lgc{GK^F}\}$ as follows
 \begin{center}
\begin{tabular}{rcl}
$\mdl{L} \sor$ 	& \quad iff \quad 	& for all Kripke models  $\langle W,R,V \rangle$ for $\lgc{L}$ and $x \in W$,\\ 
				&			& $V(A,x) \ltoe V(B,x)$ for some $(A \ltoe B) \in \sor$
\end{tabular}
\end{center}
and say in this case that $\sor$ is \emph{$\lgc{L}$-valid}.

We call a sequent of relations \emph{propositional} if it contains no occurrences of modalities, and 
identify the {\em modal part} of a sequent of relations $\sor$ as the relations in $\sor$ made up of box formulas, diamond 
formulas, and constants, calling $\sor$ {\em purely modal} if it coincides with its modal part. 
We say that a formula \emph{occurs} in a sequent of relations if it occurs as 
the left or right side of one of the relations. A sequent of relations is said to be \emph{propositionally $\lgc{L}$-valid} for 
$\lgc{L} \in \{\lgc{GK},\lgc{GK^F}\}$ 
if the sequent of relations obtained by replacing each occurrence of a box formula $\mbx A$ with a new variable $p_A$ 
and diamond formula $\dmd A$ with a new variable $q_A$ is $\lgc{G}$-valid.

It will also be helpful to adopt the following notation for sets of relations:
\[
\begin{array}{rcl}
[A_1,\ldots,A_n] \ltoe B 	& \isdef & A_1 \ltoe B \h \ldots \h A_n \ltoe B\\[0in]
[ ] \le B					& \isdef & \top \le B\\[0in]
[ ] < B					& \isdef & \emptyset\\[0in]
A \ltoe [B_1,\ldots,B_m] 	& \isdef & A \ltoe B_1 \h \ldots \h A \ltoe B_m\\[0in]
A \le []					& \isdef & A \ltoe \bot\\[0in]
A < []					& \isdef & \emptyset.
\end{array}
\]
Note that we always restrict expressions $\Ga \ltoe \De$ to cases where either $\Ga$ or $\De$ 
has at most one element. We remark, moreover, that $\Ga$ and $\De$ can also be considered sets of 
formulas rather than multisets without changing the meaning of the notation.

\begin{exa}
For instance,
\[
[\mbx (p \to q), r \to \mbx \bot] \le \mbx p \h [] \le \mbx (p \land q) \h \mbx (p \to \bot) < [p,r]
\]
stands for the sequent of relations
\[
\mbx (p \to q) \le \mbx p \h  r \to \mbx \bot \le \mbx p\h \top \le \mbx (p \land q) \h \mbx (p \to \bot) < p \h \mbx (p \to \bot) < r.
\]
\end{exa}

The logical rules, $\comr$, $\wlr$, and $\wrr$ are $\lgc{L}$-invertible for 
$\lgc{L} \in \{\lgc{GK},\lgc{GK^F}\}$. Hence applying the logical rules upwards to an $\lgc{L}$-valid sequent of relations 
terminates  with $\lgc{L}$-valid sequent of relations containing only modal formulas and atoms. 
Since sequents of relations are sets of pairs of formulas, there is a 
finite number that can be obtained by applying the rules backwards to any given sequent of relations. Hence:

\begin{lem} \label{lemsaturated}
Every sequent of relations $\sor$ is derivable from a set of saturated sequents of relations $\{\sor_1,\ldots,\sor_n\}$ 
using the logical rules, $\comr$, $\csr$, $\wlr$, and $\wrr$; moreover, $\sor$ is $\lgc{L}$-valid for $\lgc{L} \in \{\lgc{GK},\lgc{GK^F}\}$ 
iff $\sor_i$ is $\lgc{L}$-valid for $i = 1 \ldots n$. \qed
\end{lem}


\section{The Box Fragment}

\subsection{The Sequent of Relations Calculus $\lgc{SGK_\mbx}$} 

For convenience, let us assume for this section that all notions (formulas, rules, etc.) refer 
exclusively to the sublanguage ${\mathcal L}_\mbx$. We define the calculus $\lgc{SGK_\mbx}$ as the 
extension of $\lgc{SG}$ (for ${\mathcal L}_\mbx$) with the modal rule
\[\small
\begin{array}{c}
\infer[\mmr]{\sor \h \mbx \Ga \le \mbx B \h \mbx \Pi \le \bot}{\Ga \le B \h \Pi \le \bot} 
\end{array}
\]
We note that it will emerge that this calculus is sound and complete for both $\lgc{GK}$ and 
$\lgc{GK^F}$, and hence that the box fragments of these logic coincide.

\begin{exa}
Consider the following derivation of a sequent of relations corresponding to the 
axiom $\mbx(p \to q) \to (\mbx p \to \mbx q)$:
\[
\small
\infer[{(\le\to)}]{\top \le \mbx(p \to q) \to (\mbx p \to \mbx q)}{
 \infer[{\ewr}]{\mbx (p \to q) \le \mbx p \to \mbx q \h \top \le \mbx p \to \mbx q}{
  \infer[{(\le\to)}]{\mbx (p \to q) \le \mbx p \to \mbx q}{
   \infer[{\mmr}]{\mbx p \le \mbx q \h  \mbx( p \to q) \le \mbx q}{
    \infer[{(\to\le)}]{p \le q \h p \to q \le q}{
     \infer[{\ewr}]{p \le q \h \top \le q \h q < p}{
      \infer[{\comr}]{p \le q \h q < p}{
       \infer[{\idr}]{p \le q \h q < p \h p \le p}{} &
       \infer[{\idr}]{p \le q \h q < p \h q \le q}{}}} &
     \infer[{\idr}]{p \le q \h q \le q}{}}}}}}
\]
\end{exa}

We prove soundness with respect to Kripke models for $\lgc{GK^F}$, recalling that these 
include the Kripke models for $\lgc{GK}$ as a special case.

\begin{thm} \label{thmsoundsgk}
If $\der{SGK_\mbx} \sor$, then $\mdl{GK^F} \sor$.
\end{thm}
\proof
The proof is a  standard induction on the height of a derivation in $\lgc{SGK_\mbx}$. Let us just 
check that the rule $\mmr$ preserves validity in Kripke models for $\lgc{GK^F}$.  Note first that it 
suffices, using the equivalence in $\lgc{GK^F}$ of $(\mbx A_1 \to B) \lor (\mbx A_2 \to B)$ and 
$\mbx (A_1 \land A_2) \to B$, to check that $\mdl{GK^F} A \le B \h C \le \bot$ implies  $\mdl{GK^F} \mbx A \le \mbx B 
\h \mbx C \le \bot$. Suppose, contrapositively, that $\not \mdl{GK^F} \mbx A \le \mbx B 
\h \mbx C \le \bot$. I.e., for some Kripke model $K = \langle W,R,V \rangle$ for $\lgc{GK^F}$ and 
$x \in W$: $V(\mbx A,x) > V(\mbx B,x)$ and $V(\mbx C,x) > 0$. So for some  $y \in W$: 
$Rxy \to_\lgc{G} V(A,y) > Rxy \to_\lgc{G} V(B,y)$ and $Rxy \to_\lgc{G} V(C,y) > 0$. 
But this implies that $V(A,y) > V(B,y)$ and $V(C,y) > 0$. So $\not \mdl{GK^F} A \le B \h C \le \bot$ 
as required. 
\qed

Completeness is of course more complicated. The challenge is to show that a $\lgc{GK}$-valid saturated sequent of relations 
$\sor$ is derivable in $\lgc{SGK_\mbx}$. This will  also show that a $\lgc{GK^F}$-valid saturated sequent of relations 
$\sor$ is derivable in $\lgc{SGK_\mbx}$. Our strategy will be to use Lemmas \ref{lempropvalid} and \ref{lemle} to show that \emph{either} 
$\sor$ is derivable in $\lgc{SG}$ \emph{or} the modal part of $\sor$ is itself $\lgc{GK}$-valid. For the latter case,  $\sor$ is derivable 
using $\mmr$ from a less complex sequent of relations shown to be $\lgc{GK}$-valid in Lemma~\ref{lemmodal}. An inductive argument 
will then complete the proof.

\begin{lem} \label{lempropvalid}
If $\sor$ is saturated and $\lgc{GK}$-valid, then either $\sor$ is propositionally valid or 
the modal part of $\sor$ is $\lgc{GK}$-valid.
\end{lem}
\proof
Proceeding contrapositively, suppose that a saturated sequent of relations $\sor$ is not propositionally valid and the modal part 
$\sor_\mbx$ of $\sor$ is not $\lgc{GK}$-valid. Hence for some Kripke model $K = \langle W,R,V \rangle$ for 
$\lgc{GK}$ and $x \in W$: 
$V(A,x) \rtoe V(B,x)$ for each $(A \ltoe B) \in \sor^\mbx$. For each $\mbx A$ occurring in $\sor$, add a constant 
$c_A$ to the language so that $r_{c_A} = V(\mbx A, x)$. 
Let $\sor^P$ be $\sor$ with each $\mbx A$ occurring in $\sor$ replaced by $c_A$.\\[.1in]
\emph{Claim:} $\sor^P$ is not $\lgc{G}$-valid.\\[.1in]
Observe that the result follows from this claim. Let $v : \fml_{\mathcal{L}_{\lgc{G}}} \to [0,1]$ be the propositional counter-valuation for 
$\sor^P$.  Define $K' = \langle W \cup \{x_0\},R',V' \rangle$ where:
\begin{enumerate}[(1)]
\item	$R' = R \cup \{(x_0,y) \mid (x,y) \in R\}$
\item	$V'$ is $V$ extended with $V'(p,x_0) = v(p)$ for each $p \in \var$.
\end{enumerate}
Then $V'(\mbx A,x_0) = V(\mbx A,x) = r_{c_A}$ for all $\mbx A$ occurring in $\sor$. 
So, since $v$ is a counter-valuation for $\sor^P$, we have that $\sor$ is not $\lgc{GK}$-valid.\\[.1in]
\emph{Proof of claim.} 
Suppose that  $\sor^P$ is $\lgc{GK}$-valid. 
Then since $\sor^P$ is saturated, by Lemma~\ref{lem:saturatedvalid}, 
$\sor^P$ contains either $(a \le a)$, or $(c \ltoe d)$ for constants $c,d$ such that $r_C \ltoe r_D$; i.e., 
we have one of the following situations:
\begin{enumerate}[(i)]
\item	$\sor^P$ contains $a \le a$ or $\bot \le a$ or $a \le \top$ or $\bot < \top$. But then $\sor$ is propositionally 
		valid, a contradiction. 
\item	$\sor^P$ contains  $c_C \ltoe c_D$ or $c_C < \top$ or $\bot < c_D$. But since $\sor_\mbx$ is not $\lgc{L}$-valid, we must 
		have, respectively that $r_{c_C} \rtoe r_{c_D}$ or $r_{c_C} = 1$ or $0 = r_{c_D}$, a contradiction.
\item	$\sor^P$ contains $a \le c_D$ and $r_{c_D} = 1$. But then by $\wlr$, $\sor$ contains 
		$\top \le \mbx D$ and $r_{c_D} < 1$, a contradiction.
\item	$\sor^P$ contains $c_C \le a$ and $r_{c_C} = 0$. But then by $\wrr$, $\sor$ contains 
		$\mbx C \le \bot$ and $r_{c_C} > 0$, a contradiction. \qed
\end{enumerate}

\begin{lem} \label{lemle}
Let $\sor \h \sor'$ be saturated, purely modal, and $\lgc{GK}$-valid, where $\sor'$ consists only of relations of the form $A<B$. 
Then $\sor$ is  $\lgc{GK}$-valid.
\end{lem}
\proof
Let us define $\comr'$ as $\comr$ restricted to instances where $\ltoe_1$ is $\le$, and 
say that a  sequent of relations $\sor$ is {\em semi-saturated} if whenever $\sor$ occurs 
as the conclusion of a logical rule, $\comr'$, $\csr$, $\wlr$, or $\wrr$, then $\sor$ also 
occurs as one of the premises. It is then sufficient to prove the following:\\[.1in]
{\em Claim.} If $\sor \h A < \mbx B$ is semi-saturated, purely modal, and $\lgc{GK}$-valid, then $\sor$ is  $\lgc{GK}$-valid.\\[.1in]
We just notice that if $\sor \h A < \mbx B$ is semi-saturated, then $\sor$ is also semi-saturated. Hence we can 
apply the claim repeatedly, noting that the case where $\bot$ appears on the right is trivial and when $\top$ 
appears on the right, we can replace $\top$ with $\mbx \top$ and apply the claim.\\[.1in]
{\em Proof of claim.} 
Proceeding contrapositively, suppose that $\not \mdl{GK} \sor$. Then there is a Kripke model $K = \langle W,R,V \rangle$ 
for $\lgc{GK}$ and $x \in W$ 
such that $V(C,x) \rtoe V(D,x)$ for all $(C \ltoe D) \in \sor$. Moreover, if $V(A,x) \ge V(\mbx B,x)$, then  
$\not \mdl{GK} \sor \h A < \mbx B$ as required, so assume that
\begin{enumerate}[(11)]
\item[($\star$)] $V(A,x) < V(\mbx B,x)$.
\end{enumerate}
Since $\sor \h A < \mbx B$ is semi-saturated, for each  $(C \le D) \in \sor$:\\[.1in]
\begin{tabular}{rcl}
\emph{either} 	& \quad & $(C \le \mbx B) \in \sor$ and so $V(C,x) > V(\mbx B,x)$ \\[.05in]
\emph{or}		& & $(A \le D) \in \sor$ and so $V(A,x) > V(D,x)$.
\end{tabular}\\[.1in]
In particular:
\begin{enumerate}[(11)]
\item[($\star\star$)] $V(A,x) \le V(D,x) < V(C,x) \le V(\mbx B,x)$ is not possible.
\end{enumerate}
We have two cases. 	If $V(\mbx B,x) = 1$, then from the above either/or distinction, 
for each  $(C \le D) \in \sor$: $(A \le D) \in \sor$ and $V(A,x) > V(D,x)$. Hence also, since 
$\sor$ is non-empty (containing at least $\top \le \bot$), $V(A,x) > 0$. 
Now, using Lemma~\ref{lemscaleone}, we obtain a Kripke model $K' = \langle W,R,V'\rangle$ for $\lgc{GK}$ such that  
$V'(C,y) = V(A,x) \to_\lgc{G} V(C,y)$ for all $y \in W$ and $C \in \fml_{{\mathcal L}_\mbx}$. 
In particular, $V'(A,x) = V'(\mbx B,x) = 1$. Moreover, for all $(C \ltoe D) \in \sor$,  
we have $V'(C,x) = V(A,x) \to_\lgc{G} V(C,x) \ge V(C,x) \rtoe V(D,x) =V(A,x) \to_\lgc{G} V(D,x) = V'(D,x)$. 
Hence  $\not \mdl{GK} \sor \h A < \mbx B$ as required. 

  Now suppose that $V(\mbx B,x) < 1$. 
Using Lemma~\ref{lemscaletwo}, we choose a suitable automorphism of $[0,1]$ and 
define for each $i \in \mathbb{Z}^+$, a Kripke model $K_i= \langle W_i,R_i,V_i \rangle$ 
for $\lgc{GK}$ such that:
\begin{enumerate}[(1)]
\item
$\langle W_i,R_i \rangle$ is a copy of $\langle W,R \rangle$ with distinct worlds for each $i \in \mathbb{Z}^+$ where 
	$x_i$ is the corresponding copy of $x$.
\item For all formulas $E$ satisfying $V_i(A, x_i) < V_i(E,x_i) \le V_i(\mbx B,x_i)$:
		\[
		V_i(A, x_i) < V_i(E,x_i) < V(A,x_i) + 1/i.
		\]
\end{enumerate}
Now we define a model $K' = \langle W',R',V' \rangle$ such that:
\begin{enumerate}[(1)]
\item $W' = \{x_0\} \cup \bigcup_{i \in \mathbb{Z}^+} W_i$
\item	$R' =  \{(x_0,y) \mid (x_i,y) \in V_i \textrm{ for some } i \in \mathbb{Z}^+\}\cup \bigcup_{i \in \mathbb{Z}^+} R_i$
\item	$V'(p,y) = V_i(p,y)$ for all $y \in W_i$ and $V'(p,x_0) = 0$.
\end{enumerate}
But then:
\[
\begin{array}{crcl}
 & V'(\mbx B,x_0) & = & \inf(\{1\} \cup \{V'(B,y) \mid R'x_0y\})\\
		&	& = & \inf \{V_i(\mbx B,x_i) \mid i \in \mathbb{Z}^+\}\\
		&	& = & V'(A,x_0).
\end{array}
\]
Clearly, if $V(C,x) \ge V(D,x)$ for some $(C < D) \in \sor$ where $C,D$ are box formulas or constants, 
then $V'(C,x_0) \ge V'(D,x_0)$.  Moreover, if $(C \le D) \in \sor$ where $C,D$ are box formulas or constants, 
then using $(\star\star)$, it follows that $V'(C,x_0) > V'(D,x_0)$. Hence  $\not \mdl{GK} \sor \h A < \mbx B$ as required. 
\qed

\begin{lem} \label{lemmodal}
If $\mdl{GK} \{\mbx A_i \le \mbx B_i\}_{i=1}^n \h \mbx C \le \bot$, then 
\[
\mdl{GK}  \bigwedge_{j \in J} A_j \le \bigwedge_{j \in J} B_j \h C \le \bot
\]
for some $\emptyset \subset J \subseteq \{1,\ldots,n\}$.
\end{lem}
\proof
We argue by contraposition; i.e., suppose that:
\[
\not \mdl{GK} \bigwedge_{j \in J} A_j \le \bigwedge_{j \in J} B_i \h  C \le \bot \quad \textrm{ for all }
\emptyset \subset J \subseteq \{1,\ldots,n\}.
\]
In particular for $i = 1 \ldots n$:
\[
\not \mdl{GK} A_i \land \ldots \land A_n \le B_i \land \ldots \land B_n \h  C \le \bot.
\]
So for each  $i = 1 \ldots n$, 
there exists a Kripke model $K_i = \langle W_i,R_i,V_i \rangle$ for $\lgc{GK}$ and  $x_i \in W_i$ (with each $W_i$ distinct) such that:
\[
V_i(A_i \land \ldots \land A_n,x_i) > V_i(B_i \land \ldots \land B_n,x_i) \quad 
\textrm{ and } \quad V_i(C,x_i) > 0.
\]
Moreover, using Lemma~\ref{lemscaletwo}, we can assume without loss of generality that
\[
V_i(B_i,x_i) \le V_i(B_k,x_i) \quad \textrm{ and so } \quad V_i(A_k,x_i) > V_i(B_i,x_i)  \quad \textrm{ for }k = i \ldots n.
\]
Now, again using Lemma~\ref{lemscaletwo}, we define iteratively 
$K'_i = \langle W_i,R_i,V'_i \rangle$ for $i = n \ldots 1$ such that for $j = i \ldots n$:
\begin{enumerate}[(i)]
\item	$V'_j(B_j,x_j)  < V_k(A_j,x_k)$ for $k = 1 \ldots i-1$.
\item    $V'_j(B_j,x_j)  < V'_k(A_j,x_k)$ for $k = i \ldots n$.
\item	$V'_j(C,x_j) > 0$.
\end{enumerate}
Let us deal with step $i$, supposing that we have already dealt with steps $n \ldots i+1$. 
We choose an order automorphism $h$ of $[0,1]$ scaling the interval $[0,V_i(B_i,x_i)]$ to a smaller 
interval $[0,V'_i(B_i,x_i)]$ so that $V'_i(B_i,x_i)=h(V(B_i,x_i))$ satisfies $V'_i(B_i,x_i)  < V_k(A_i,x_k)$ for $k = 1 \ldots i-1$ and 
$V'_i(B_i,x_i)  < V'_k(A_i,x_k)$ for $k = i \ldots n$. This is  possible since we need only force the value 
$V'_i(B_i,x_i)$ to be suitably small. Also clearly $V'_i(C,x_i) > 0$. Now consider $j \in \{i+1,\ldots,n\}$. 
By construction, we already have $V'_j(B_j,x_j)  < V_k(A_j,x_k)$ for $k = 1 \ldots i-1$, 
$V'_j(B_j,x_j)  < V'_k(A_j,x_k)$ for $k = i+1 \ldots n$, and $V'_j(C,x_j) > 0$. So it remains only to show 
that $V'_j(B_j,x_j)  < V'_i(A_j,x_i)$.  Note first that using step $j > i$, we have that 
$V'_j(B_j,x_j) < V_i(A_j,x_i)$. But  $V_i(A_j,x_i) > V_i(B_i,x_i)$, so 
we can assume that  $V'_i(A_j,x_i) = V_i(A_j,x_i)$. Hence 
$V'_j(B_j,x_j) < V'_i(A_j,x_i)$ as required.

Finally, we define a model $K = \langle W, R, V \rangle$ where for a new world $x_0$:
\begin{enumerate}[(1)]
\item	$W = W_1 \cup \ldots \cup W_n \cup \{x_0\}$.
\item	$R = R_1 \cup \ldots \cup R_n \cup \{(x_0,x_1), \ldots, (x_0,x_n)\}\}$.
\item	$V(p,x) = V'_i(p,x)$ for all $x \in W_i$ and $V(p,x_0) = 0$.
\end{enumerate}
But then for $i= 1 \ldots n$:
\[
V(\mbx B_i,x_0) \le V'_i(B_i,x_i) < V'_j(A_i,x_j)  \quad {\rm for }\ j = 1 \ldots n.
\]
So $V(\mbx B_i,x_0)  < V(\mbx A_i,x_0)$. Since also, using (iii), $V(\mbx C,x_0) > 0$, we 
have $\not \mdl{GK} \{\mbx A_i \le \mbx B_i\}_{i=1}^n \h \mbx C \le \bot$ as required.   \qed

\begin{thm} \label{thm:boxcomp}
If $\mdl{GK} \sor$, then $\der{SGK_\mbx} \sor$.
\end{thm}
\proof
We prove the theorem by induction on the modal degree of the sequent of relations $\sor$: the maximal complexity 
of a boxed subformula occurring in $\sor$. If the modal degree is $0$, then $\sor$ is propositional and $\lgc{SGK_\mbx}$-derivable. 
Otherwise, by Lemma~\ref{lemsaturated} (since working upwards, the rules do not increase modal degree), we can assume that 
$\sor$ is both $\lgc{GK}$-valid and saturated. If $\sor$ is propositionally $\lgc{GK}$-valid, then it is derivable. Otherwise, by 
Lemmas~\ref{lempropvalid} and~\ref{lemle}, $\sor$ is of the form:
\[
\sor' \h \{\mbx A_i \le \mbx B_i\}_{i=1}^n \h \{\mbx C_j \le \bot\}_{j=1}^m
\]
and $\mdl{GK} \{\mbx A_i \le \mbx B_i\}_{i=1}^n \h \{\mbx C_j \le \bot\}_{j=1}^m$. But then by Lemma~\ref{lemmodal}:
\[
\mdl{GK}  \bigwedge_{i \in I} A_i \le \bigwedge_{i \in I} B_i \h \{C_j \le \bot\}_{j=1}^m \quad \textrm{ for some $\emptyset \subset I \subseteq \{1,\ldots,n\}$.}
\]
Let us assume without loss of generality that $I = \{1,\ldots,n\}$. Then also:
\[
\mdl{GK} \{A_i \le B_k\}_{i=1}^n \h \{C_j \le \bot\}_{j=1}^m \quad \textrm{for }k=1\ldots n.
\]
So by the induction hypothesis and an application of $\mmr$:
\[
\der{SGK} \sor' \h \{\mbx A_i \le \mbx B_k\}_{i=1}^n \h \{\mbx C_j \le \bot\}_{j=1}^m \quad \textrm{for }k=1\ldots n.
\]
But then $\sor$ is derivable by repeated applications of $\comr$. \qed

\begin{cor}
$\der{\lgc{SGK_\mbx}} \sor$ iff $\mdl{GK^{F}} \sor$ iff  $\mdl{GK} \sor$. \qed
\end{cor}


\subsection{Consequences}

We can use the above characterization of $\lgc{SGK_\mbx}$ to give an alternative proof of completeness of the 
axiomatization $\lgc{HGK_\mbx}$ with respect to both standard and fuzzy Kripke frames, recalling that this result was 
first obtained by Caicedo and Rodr{\'i}guez in~\cite{CR08}.

\begin{thm} \label{thmaxcomp}
$\der{\lgc{HGK_\mbx}} A$ iff $\mdl{GK^{F}} A$ iff  $\mdl{GK} A$.
\end{thm}
\proof
We can easily show that the axioms of $\lgc{HGK_\mbx}$ are $\lgc{GK^F}$-valid and that the 
rules $\mdp$ and $\nec_\mbx$ preserve $\lgc{GK^F}$-validity. So $\der{\lgc{HGK_\mbx}} A$ implies $\mdl{GK^{F}} A$, 
which, since all Kripke frames are fuzzy Kripke frames, implies $\mdl{GK} A$. On the other hand, if
$\mdl{GK} A$, then by Theorem~\ref{thm:boxcomp}, $\der{SGK_\mbx} \top \le A$. Hence it suffices to show that $\der{SGK_\mbx} \sor$ 
implies $\der{HGK_\mbx} I_S(\sor)$. I.e., we  need  that for each rule $\sor_1,\ldots,\sor_n \ / \ \sor$ of $\lgc{SGK_\mbx}$, 
whenever $\der{\lgc{HGK_\mbx}} I_S(\sor_i)$ for $i = 1 \ldots n$, also $\der{\lgc{HGK_\mbx}} I_S(\sor)$.
This is straightforward  for the logical rules and easy for the axioms and structural rules, so let us just consider the case of 
$\mmr$. Using the derivability $\der{\lgc{HGK_\mbx}} ((\mbx A_1 \to B) \lorÊ(\mbx A_2 \to B)) \leftrightarrow (\mbx (A_1 \land A_2) \to B)$, it 
is enough to show that $\der{\lgc{HGK_\mbx}} (A \to B) \lor \lnot C$ implies $\der{\lgc{HGK_\mbx}} (\mbx A \to \mbx B) \lor \lnot \mbx C$. Note first that $\der{\lgc{HG}} (\lnot \lnot F \to G) \leftrightarrow (G \lor \lnot F)$. Hence  if $\der{\lgc{HGK_\mbx}} (A \to B) \lor \lnot C$, then $\der{\lgc{HGK_\mbx}} \lnot \lnot C \to (A \to B)$. But then by $\nec$, $\der{\lgc{HGK_\mbx}} \mbx (\lnot \lnot C \to (A \to B))$ and using $(K_\mbx)$, $\der{\lgc{HGK_\mbx}} \mbx \lnot \lnot C \to (\mbx A \to \mbx B)$. Hence, using $(Z_\mbx)$, $\der{\lgc{HGK_\mbx}} \lnot \lnot \mbx C \to (\mbx A \to \mbx B)$, and 
finally, $\der{\lgc{HGK_\mbx}} (\mbx A \to \mbx B) \lor \lnot \mbx C$ as required. \qed

Our completeness proof for $\lgc{SGK_\mbx}$ can also be exploited to obtain a precise 
bound for the complexity of the $\lgc{GK}$-validity problem for the box fragment, namely the problem 
of checking $\mdl{\lgc{GK}} A$ for any $A \in \fml_{{\mathcal L}_\mbx}$.

\begin{thm} \label{thm:boxpspace}
The validity problem for the box fragment of $\lgc{GK}$ is PSPACE-complete.
\end{thm}
\proof
First we show that checking $\lgc{GK}$-validity for the box fragment is PSPACE-hard. We recall that the modal logic 
$\lgc{K}$ is PSPACE-complete  (see, e.g.,~\cite{cha:mod}). Consider  the translation $*$ sending each  propositional variable 
$p$ to its double negation $\lnot \lnot p$. We can easily show  that $\mdl{K} A$ iff $\mdl{GK} A^*$ which establishes that 
the validity problem for the box fragment of $\lgc{GK}$ must also be PSPACE-hard. For the non-trivial direction consider any Kripke model  
$\langle W, R, V\rangle$ for $\lgc{GK}$ and define a standard Kripke model $\langle W, R, V' \rangle$ by stipulating:  $V'(p, x) = V(\lnot \lnot p, x)$. Then by a simple induction, $V'(C, x) = V(C^*, x) \in \{0, 1\}$ for any  $C \in \fml_{{\mathcal L}_\mbx}$. 
Hence if $\not \mdl{GK} A^*$, then $\not \mdl{GK} A$.

For PSPACE-inclusion, we consider derivations in the sequent of relations calculus $\lgc{SGK_\mbx}$. 
Given a formula $A \in \fml_{{\mathcal L}_\mbx}$, let $Sub(A)$ be the set of subformulas of $A$ together with 
$\top, \bot$, and consider the set 
$$\Phi_A= \{C \ltoe D \mid  C, D \in Sub(A), \ltoe \in \{<,\le\} \}.$$
The cardinality of $\Phi_A$  is $O(|A|^2)$.
Since any sequent of relations appearing in a derivation of $\top \le A$ is a subset
of $\Phi_A$, its size is also $O(|A|^2)$. 

We now consider the length of branches in the search for a derivation of $\top \le A$ in $\lgc{SGK_\mbx}$. 
Using the $\lgc{GK}$-invertibility of the logical rules we assume that any branch of a derivation is expanded
 by applying iteratively the rules upwards in the following order:
\begin{enumerate}[(1)]
\item[(1)] Apply the logical rules, $\wlr$, $\wrr$, $\csr$, and $\comr$ in order to obtain a saturated sequent and check the axioms.
\item[(2)] Apply $\mmr$ and restart from (1) with the premise of this rule. 
\end{enumerate}
The length of the branch built in (1) is $O(|A|^2)$ since each logical rule replaces one relation with one or two relations involving formulas of smaller complexity, and each application of $\wlr$, $\wrr$, $\csr$, and $\comr$ add exactly one relation at a time, with the total number of different relations possible being $O(|A|^2)$. The sequent obtained in (2) by applying $\mmr$ has a smaller or equal number of relations and a strictly smaller modal degree. The entire length of a proof branch is hence bounded by $O(|A|^2 \times m)$ = $O(|A|^3)$, where $m$ is the modal degree of $A$. 

Thus storing a branch of a proof requires only polynomial space. Moreover, the
branching is at most binary.  As usual, we  search for a proof in a depth-first
manner: we  store one  branch at a time together with some information (requiring a
small amount of space, say logarithmic space) to reconstruct branching points  and
backtracking points, the latter determined by alternative applications of $\mmr$. 
Hence the total amount of space needed for carrying out proof search
is polynomial in  $|A|$, and so deciding validity for the box fragment of $\lgc{GK}$  is in PSPACE.  \qed

Note that the validity problem for modal finite-valued G\"odel logics (in the full language) is also 
PSPACE-complete. This result was established in~\cite{Bobillo09} using a reduction to the classical case 
for the equivalent problem for description logics based on finite-valued G\"odel logics.


\subsection{A Hypersequent Calculus}

A more elegant analytic calculus for the box fragment -- lacking, however, the invertible logical rules of 
$\lgc{SG}$ -- can be presented in the framework of hypersequents. In particular, let us define the hypersequent 
calculus $\lgc{GGK_\mbx}$ as  the extension of the calculus $\lgc{GG}$ given in Figure~\ref{figgod} (for ${\mathcal L}_\mbx$) 
with the modal rule:
\[\small
\infer[\mmr]{\mbx \Pi \seq \h \mbx \Ga \seq \mbx A}{\Pi \seq \h \Ga \seq A}
\]
$\mmr$ is  a version of the ordinary Gentzen rule for the (classical) modal logic $\lgc{K}$, obtained by adding the extra sequent $\Pi \seq$ in the premise 
and $\mbx \Pi \seq$ in the conclusion. This extra component reflects the fact that $\bot$ is interpreted as the bottom element $0$ in 
each world. 
		
\begin{exa}
All the axioms of $\lgc{HGK_\mbx}$ are derivable in $\lgc{GGK_\mbx}$; e.g., for $(Z_\mbx)$:
\[\small
\infer[\ilr]{\seq \lnot \lnot \mbx p \to \mbx \lnot \lnot p}{
 \infer[{\ecr}]{\lnot \lnot \mbx p \seq \mbx \lnot \lnot p}{
   \infer[\wrr]{\lnot \lnot \mbx p \seq \mbx \lnot \lnot p \h \lnot \lnot \mbx p \seq \mbx \lnot \lnot p}{
    \infer[\wlr]{\lnot \lnot \mbx p \seq \h \lnot \lnot \mbx p \seq \mbx \lnot \lnot p}{
      \infer[\nlr]{\lnot \lnot \mbx p \seq \h \seq \mbx \lnot \lnot p}{
        \infer[\nrr]{\seq \lnot \mbx p  \h \seq \mbx \lnot \lnot p}{
      	\infer[\mmr]{\mbx p \seq \h  \seq \mbx \lnot \lnot p}{
	 \infer[\nrr]{p \seq \h \seq \lnot \lnot p}{
	  \infer[\clr]{p \seq \h \lnot p \seq}{
	   \infer[\clr]{p, p \seq \h \lnot p \seq}{
	    \infer[{\comr}]{p, p \seq \h \lnot p, \lnot p \seq}{
	   \infer[\nlr]{p, \lnot p \seq}{
	    \infer[{\idr}]{p \seq p}{}} & 
	       \infer[\nlr]{p, \lnot p \seq}{
	    \infer[{\idr}]{p \seq p}{}}}}}}}}}}}}}
	\]
\end{exa}

It will be helpful for proving cut-elimination to consider the following generalizations of the rule $\mmr$, derivable using 
$\comr$, $\mmr$, $\clr$, and $\wlr$:
		\[\small
		\infer[\mmr^n]{\mbx \Pi_1 \seq \h \ldots \h \mbx \Pi_n \seq \h \mbx \Ga \seq \mbx A}{
		 \Pi_1 \seq \h \ldots \h \Pi_n \seq \h \Ga \seq A} \qquad (n \in \mathbb{N})
		\]
For example, in the case of $n=2$, we have the derivation:
\[\small
\infer[\clr]{\mbx \Pi_1 \seq \h \mbx \Pi_2 \seq \h \mbx \Ga \seq \mbx A}{
 \infer[\clr]{\vdots}{
 \infer[\comr]{\mbx \Pi_1, \mbx \Pi_1 \seq \h \mbx \Pi_2, \mbx \Pi_2 \seq \h \mbx \Ga \seq \mbx A}{
  \infer[\mmr]{\mbx \Pi_1, \mbx \Pi_2 \seq \h \mbx \Ga \seq \mbx A}{
   \infer[\ecr]{\Pi_1, \Pi_2 \seq \h \Ga \seq A}{
       \infer[\wlr]{\Pi_1, \Pi_2 \seq \h \Pi_1, \Pi_2 \seq \h \Ga \seq A}{
        \infer[\wlr]{\vdots}{
         \Pi_1 \seq \h \Pi_2 \seq \h \Ga \seq A}}}} & 
  \infer[\mmr]{\mbx \Pi_1, \mbx \Pi_2 \seq \h \mbx \Ga \seq \mbx A}{
   \infer[\ecr]{\Pi_1, \Pi_2 \seq \h \Ga \seq A}{
       \infer[\wlr]{\Pi_1, \Pi_2 \seq \h \Pi_1, \Pi_2 \seq \h \Ga \seq A}{
        \infer[\wlr]{\vdots}{
         \Pi_1 \seq \h \Pi_2 \seq \h \Ga \seq A}}}}}}}          
\]

\begin{thm}
$\der{\lgc{GGK_\mbx}} \mg$ iff $\mdl{GK} I_H(\mg)$.
\end{thm}
\proof
The left-to-right direction (soundness) is proved as usual by induction on the height of a derivation in 
$\lgc{GGK_\mbx}$. For the right-to-left direction (completeness), we make use of the completeness of 
the axiom system $\lgc{HGK_\mbx}$ established in Theorem~\ref{thmaxcomp}. I.e., 
$\mdl{GK} I_H(\mg)$ implies $\der{\lgc{HGK_\mbx}} I_H(\mg)$. But now, since all the axioms of 
$\lgc{HGK_\mbx}$ are derivable in $\lgc{GGK_\mbx}$ and the rules $\mdp$ and $\nec_\mbx$ are also derivable, 
we have that $\der{\lgc{HGK_\mbx}} I_H(\mg)$ implies $\der{\lgc{GGK_\mbx}} I_H(\mg)$. Finally, it is 
straightforward to show (following~\cite{met:book}, Proposition~4.61) that  $\der{\lgc{GGK_\mbx}} I_H(\mg)$ 
implies $\der{\lgc{GGK_\mbx}} \mg$ as required. \qed

Let us show now that {\em  cut-elimination} holds for $\lgc{GGK_\mbx}$, i.e., that there is a constructive procedure for transforming a derivation of a 
hypersequent $\mg$ in this calculus into a derivation of $\mg$ with no applications of $\cutr$. We write $d \der{S} X$ to denote that $d$ is a 
derivation of  $X$ in a calculus $\lgc{S}$ and $|d|$ for the height of the derivation considered as a tree. We also recall that 
the {\em principal formula} of an application of a rule is the distinguished formula in the conclusion and that the {\em cut-formula} of an application of 
$\cutr$ is the formula appearing in the premises but not the conclusion.

\begin{thm}
Cut-elimination holds for $\lgc{GGK_\mbx}$.
\end{thm}
\proof
Let $\lgc{GGK_\mbx^\circ}$ be $\lgc{GGK_\mbx}$ with $\cutr$ removed. Then to establish cut-elimination for $\lgc{GGK_\mbx}$ 
it is sufficient to give a constructive proof of the following:\\[.1in]
\begin{tabular}{rl}
\emph{Claim.} & If $d_1 \der{\lgc{GGK_\mbx^\circ}} [\Ga_i, [A]^{\lam_i} \seq \De_i]_{i=1}^n$ and  
$d_2 \der{\lgc{GGK_\mbx^\circ}} \mh \h [\Pi_j \seq A]_{j=1}^m$,\\
& then  $\der{\lgc{GGK_\mbx^\circ}} \mh \h [\Ga_i, \Pi_j^{\lam_i} \seq \De_i]_{i=1 \ldots n}^{j = 1 \ldots m}$.\\
\end{tabular}\\[.1in]
We proceed by induction on the lexicographically ordered pair $\langle |A|,|d_1| + |d_2| \rangle$. 
If the last step in $d_1$ or $d_2$ is an axiom, then the result follows almost immediately. Also, if 
the last step in either derivation is not $\mmr$ or does not have the cut-formula $A$ as the principal formula, 
then the result follows by applications of the induction hypothesis to the premises and applications of the 
same rule and structural rules. Suppose for example that one of the derivations ends with an application of $\comr$ 
(the other derivation may end with $\mmr$):
\[
\small
\infer{\mg \h \Ga'_1, \Ga''_1, [A]^{\lam'_1+\lam''_1} \seq \De_1  \h \Ga'_2, \Ga''_2, [A]^{\lam'_2 + \lam''_2} \seq \De_2}{
\mg \h \Ga'_1, \Ga'_2, [A]^{\lam'_1 + \lam'_2} \seq \De_1 &
\mg \h \Ga''_1, \Ga''_2, [A]^{\lam''_1 + \lam''_2} \seq \De_2}
\]
where $\mg = [\Ga_i, [A]^{\lam_i} \seq \De_i]_{i=3}^n$. Then by the induction hypothesis twice:
\[
\der{\lgc{GGK_\mbx^\circ}} \mh' \h  [\Ga'_1, \Ga'_2, \Pi_j^{\lam'_1 + \lam'_2} \seq \De_1]_{j=1}^m \quad {\rm and} \quad
\der{\lgc{GGK_\mbx^\circ}} \mh' \h  [\Ga''_1, \Ga''_2, \Pi_j^{\lam''_1 + \lam''_2} \seq \De_2]_{j=1}^m 
\]
where $\mh' = \mh \h [\Ga_i, \Pi_j^{\lam_i} \seq \De_i]_{i=3 \ldots n}^{j = 1 \ldots m}$. The required hypersequent:
\[
\mh' \h  [\Ga'_1, \Pi_j^{\lam'_1 + \lam''_1}, \Ga''_1 \seq \De_1]_{j=1}^m  \h [\Ga'_2, \Pi_j^{\lam'_2 + \lam''_2}, \Ga''_2 \seq \De_2]_{j=1}^m
\]
is derivable by repeated applications of $\comr$, $\ecr$, and $\ewr$.

If a distinguished occurrence of $A$ is the principal formula in both derivations and of the form $B \land C$, $B \lor C$, or $B \to C$, 
then we can first use the induction hypothesis applied to the premises in one derivation and the conclusion in the other, and then 
apply the induction hypothesis again with cut-formulas $B$ and $C$ of smaller complexity. The result follows using applications of 
$\ecr$ and/or $\ewr$ as required. Consider then the hardest case where both derivations $d_1$ and $d_2$ end as follows with an 
application of $\mmr$ and $A$ is of the form $\mbx B$:
\[\small
\infer[\mmr]{\mbx \Ga_1, [\mbx B]^{\lam_1} \seq \h \mbx \Ga_2, [\mbx B]^{\lam_2} \seq \mbx C}{
 \infer{\Ga_1, [B]^{\lam_1} \seq \h \Ga_2, [B]^{\lam_2} \seq C}{\vdots}} \qquad
\infer[\mmr]{\mbx \Si \seq \h \mbx \Pi \seq \mbx B}{\infer{\Si \seq \h \Pi \seq B}{\vdots}}
\]
Then since $|B| < |\mbx B|$, we can apply the induction hypothesis to the $\lgc{GGK_\mbx^\circ}$-derivable hypersequents 
$(\Ga_1, [B]^{\lam_1} \seq \h \Ga_2, [B]^{\lam_2} \seq C)$ and $(\Si \seq \h \Pi \seq B)$ to obtain a $\lgc{GGK_\mbx^\circ}$-derivation of  
$(\Si \seq \h \Ga_1, \Pi^{\lam_1} \seq \h \Ga_2, \Pi^{\lam_2} \seq C)$. Hence by an application of the derived rule $\mmr^2$, we obtain a 
$\lgc{GGK_\mbx^\circ}$-derivation of $\mbx \Si \seq \h \mbx \Ga_1, \mbx \Pi^{\lam_1} \seq \h \mbx \Ga_2, \mbx \Pi^{\lam_2} \seq \mbx C$ 
as required. \qed

We note finally that a related hypersequent calculus $\lgc{GGK^r}$ was defined (along with many other such calculi) in~\cite{met:mod} 
by extending $\lgc{GG}$ with the rule:
\[
\small
\infer{\mg \h \mbx \Ga \seq \mbx A}{\mg \h \Ga \seq A}
\]
It was shown in \cite{met:mod} that $\lgc{GGK^r}$ is complete with respect to the standard semantics for G\"odel logic extended with a unary function on $[0,1]$ that can be interpreted in fuzzy logic as a ``truth stresser'' modality such as ``very true''.


\section{The Diamond Fragments}

\subsection{Sequent of Relations Calculi}

Let us turn our attention now to the (distinct) diamond fragments of $\lgc{GK}$ and $\lgc{GK^F}$, 
and for convenience assume for the rest of this section that all notions (formulas, rules, etc.)  refer 
exclusively to the sublanguage ${\mathcal L}_\dmd$. We introduce the following systems:

\begin{enumerate}[$\bullet$]

\item 	$\lgc{SGK_\dmd}$ consists of $\lgc{SG}$ (for  ${\mathcal L}_\dmd$) extended with:
		\[\small
		\infer[\mdr]{\sor \h \dmd A \le \dmd \De \h \bot < \dmd \Si \h \top \le \dmd \The}{A \le \De \h \bot < \Si \h \top \le \The} 
		\]

\item	 	$\lgc{SGK^F_\dmd}$  consists of $\lgc{SG}$  (for ${\mathcal L}_\dmd$) extended with:
		\[\small
		\infer[\mdr^*]{\sor \h \dmd A \le \dmd \De \h \bot < \dmd \Si}{A \le \De \h \bot < \Si} 
		\]
		
\end{enumerate}

\begin{exa}
Axioms from $(Z_\dmd)$ are derivable in both $\lgc{SGK_\dmd}$ and $\lgc{SGK^F_\dmd}$:
\[\small
\infer[(\le\to)]{\top \le \dmd \lnot \lnot p \to  \lnot \lnot \dmd p}{
 \infer[\ewr]{\top \le \lnot \lnot \dmd p \h \dmd \lnot \lnot p \le \lnot \lnot \dmd p}{
 \infer[(\le\!\lnot)]{\dmd \lnot \lnot p \le \lnot \lnot \dmd p}{
  \infer[(\lnot\!\le)]{\dmd \lnot \lnot p \le \bot \h \lnot \dmd p \le \bot}{
   \infer[\mdr^*]{\dmd \lnot \lnot p \le \bot \h \bot < \dmd p}{
     \infer[(\lnot\!\le)]{\lnot \lnot p \le \bot \h \bot < p}{
      \infer[(<\!\lnot)]{\bot < \lnot p \h \bot < p}{
       \infer[{\ewr}]{p \le \bot \h \bot < \bot \h \bot < p}{
        \infer[\comr]{p \le \bot \h \bot < p}{
         \infer[{\idr}]{p \le \bot \h \bot < p \h p \le p}{} & \infer[{\idr}]{p \le \bot \h \bot < p \h \bot \le \bot}{}}} & 
         \infer[(<)]{\bot < \top \h \bot < p}{}}}}}}}}
\]
However, the opposite direction is derivable only in $\lgc{SGK_\dmd}$:
\[\small
\infer[(\le\to)]{\top \le \lnot \lnot \dmd p \to  \dmd \lnot \lnot p}{
 \infer[\ewr]{\top \le   \dmd \lnot \lnot p \h \lnot \lnot \dmd p \le \dmd \lnot \lnot  p}{
 \infer[(\lnot\!\le)]{\lnot \lnot \dmd p \le \dmd \lnot \lnot  p}{
  \infer[(<\!\lnot)]{\top \le \dmd \lnot \lnot  p \h \bot < \lnot \dmd p}{
   \infer[{\ewr}]{\top \le \dmd \lnot \lnot p \h   \dmd p \le \bot \h \bot < \bot}{
   \infer[\mdr]{\top \le \dmd \lnot \lnot p \h   \dmd p \le \bot}{
     \infer[(\le\!\lnot)]{\top \le  \lnot \lnot p \h p \le \bot}{
      \infer[\ewr]{\lnot p \le \bot \h \top \le \bot \h p \le \bot}{
      \infer[(\lnot\!\le)]{\lnot p \le \bot  \h p \le \bot}{
       \infer[\comr]{\bot < p \h p \le \bot}{
        \infer[{\idr}]{\bot < p \h p \le \bot \h \bot \le \bot}{} &
        \infer[{\idr}]{\bot < p \h p \le \bot \h p \le p}{}}}}}}} & 
     \infer[(<)]{\top \le \dmd \lnot \lnot  p \h \bot < \top}{}}}}}
\]
\end{exa}

\begin{thm} \label{thmsoundsgkdf}
Let $\lgc{L} \in \{\lgc{GK},\lgc{GK^F}\}$. If $\der{SL_\dmd} \sor$, then $\mdl{L} \sor$. 
\end{thm}
\proof
Again, the proof is a straightforward induction on the height of a derivation, and 
it suffices to check that $\mdr$ and $\mdr^*$ preserve $\lgc{GK}$-validity and 
$\lgc{GK^F}$-validity, respectively. Moreover, using the distributivity of $\dmd$ over $\lor$, 
we can assume that $\De$, $\Si$, and $\The$ each contain exactly one formula. 
Hence for $\mdr$ suppose contrapositively that for some $K = \langle W,R,V \rangle$ and 
$x \in W$: $V(\dmd A,x) > V(\dmd B,x)$, $V(\dmd C,x) = 0$, and $V(\dmd D,x) < 1$. 
Then there is a world $y$ such that $Rxy$ and $V(A,y) > V(B,y)$, 
$V(C,y) = 0$, and $V(D,y) < 1$ as required.  
Now for $\mdr^*$ suppose contrapositively that for some $K = \langle W,R,V \rangle$ and $x \in W$: $V(\dmd A,x) > V(\dmd B,x)$ and 
$V(\dmd C,x) = 0$. Then for some world $y$:  $\min (V(A,y),Rxy) > Ê\min(V(B,y),Rxy)$, which implies 
$V(A,y) > V(B,y)$, and $V(\dmd C,x) = 0$. \qed

Completeness proofs mostly follow the same pattern as the proof 
for $\lgc{SGK_\mbx}$; however, there exist a couple of significant differences in the details.

\begin{lem} 
Suppose that $\sor$ is saturated and $\lgc{L}$-valid for $\lgc{L} \in \{\lgc{GK},\lgc{GK^F}\}$. 
Then either $\sor$ is propositionally $\lgc{L}$-valid or the modal part of $\sor$ is 
$\lgc{L}$-valid.
 \end{lem}
 \proof
Proceeding contrapositively, let $\sor_\dmd$ be the modal part of $\sor$ and 
suppose that $\sor_\dmd$ is not $\lgc{L}$-valid and $\sor$ is not 
propositionally $\lgc{L}$-valid. Then for  some model $K = \langle W,R,V \rangle$ for 
$\lgc{L}$ and $x \in W$:  $V(A,x) \rtoe V(B,x)$ for all $(A \ltoe B) \in \sor^\dmd$. 
For each $\dmd A$ occurring in $\sor$, let us add a constant $c_A$ 
to the language so that $r_{c_A} = V(\dmd A, x)$. 
Let $\sor^P$ be $\sor$ with each $\dmd A$ occurring in $\sor$ replaced by $c_A$. 
Exactly as in the proof of Lemma~\ref{lempropvalid}, it follows that  $\sor^P$ is not $\lgc{L}$-valid. 
But now let  $v :  \fml_{\mathcal{L}_{\lgc{G}}}  \to [0,1]$ be the propositional 
counter-valuation for $\sor^P$.  Define $K' = \langle W \cup \{x_0\},R',V' \rangle$ where:
\begin{enumerate}[(1)]
\item	$R'yz = \begin{cases} 
			Ryz  & y,z \in W\\
			Rxz & y = x_0, z \in W\\
			0	& z = x_0.
			\end{cases}$
\item	$V'$ is $V$ extended with $V'(p,x_0) = v(p)$ for all $p \in \var$.
\end{enumerate}
Then  for each $\dmd A$ occurring in $\sor$:
\[
V'(\dmd A,x_0) = \sup_{y \in W} (\min(V'(A,y),Rx_0y)) = \sup_{y \in W} (\min(V(A,y),Rxy)) = V(\dmd A,x) = r_{c_A}.
\]
So, since $v$ is a counter-valuation for $\sor^P$, we have that 
$\sor$ is not $\lgc{L}$-valid as required.  \qed

\begin{lem} \label{lemlt}
Let $\sor \h \sor'$ be saturated, purely modal, and $\lgc{L}$-valid 
for $\lgc{L} \in \{\lgc{GK},\lgc{GK^F}\}$, where $\sor'$ consists only of relations of the form $\dmd A<B$. 
Then $\sor$ is  $\lgc{L}$-valid.
\end{lem}
\proof
Recall that $\comr'$ is defined as $\comr$ restricted to instances where $\ltoe_1$ is $\le$, and 
 that an atomic sequent of relations $\sor$ is {\em semi-saturated} if whenever $\sor$ occurs 
as the conclusion of $\comr'$, $\csr$, $\wlr$, or $\wrr$, then $\sor$ also 
occurs as one of the premises. It is then sufficient to prove the following:\\[.1in]
{\em Claim.} If $\sor \h \dmd A < B$ is semi-saturated, purely modal, and $\lgc{L}$-valid for $\lgc{L} \in \{\lgc{GK},\lgc{GK^F}\}$, 
then $\sor$ is  $\lgc{L}$-valid.\\[.1in]
{\em Proof of claim.} 
Proceeding contrapositively, suppose that $\sor$ is not  $\lgc{L}$-valid. 
Then there is a Kripke model 
$K = \langle W,R,V \rangle$ for $\lgc{L}$ and $x \in W$ such that $V(C,x) \rtoe V(D,x)$ for all 
$(C \ltoe D) \in \sor$. Moreover, if $V(\dmd A,x) \ge V(B,x)$, then  $\sor \h \dmd A < B$ is not 
$\lgc{L}$-valid as required, so assume that:
\begin{enumerate}[(11)]
\item[($\star$)] \quad $V(\dmd A,x) < V(B,x)$.
\end{enumerate}
Since $\sor \h \dmd A < B$ is semi-saturated, for each  $(C \le D) \in \sor$:\\[.1in]
\begin{tabular}{rcl}
\emph{either} 	& \quad & $(C \le B) \in \sor$ and so $V(C,x) > V(B,x)$ \\[.05in]
\emph{or}		& & $(\dmd A \le D) \in \sor$ and so $V(\dmd A,x) > V(D,x)$.
\end{tabular}\\[.1in]
In particular:
\begin{enumerate}[(11)]
\item[($\star\star$)] \quad $V(\dmd A,x) \le V(D,x) < V(C,x) \le V(B,x)$ is not possible.
\end{enumerate}
Suppose that $V(\dmd A,x) = 0$. Using $\csr$ and $\comr$, either $(\dmd A \le \bot) \in \sor$ or $(\bot < B) \in \sor$. 
In the first case, $V(\dmd A, w) > 0$, a contradiction. In the second, $V(B,x) = 0$, also a contradiction.	
So let us assume $V(\dmd A,x) > 0$. Then using Lemma~\ref{lemscaletwo}, we define for each $i \in \mathbb{Z}^+$, 
a Kripke model $K_i= \langle W_i,R_i,V_i \rangle$ for $\lgc{L}$ such that:
\begin{enumerate}[(1)]

\item	$\langle W_i,R_i \rangle$ is a copy of $\langle W,R \rangle$ with distinct worlds for each $i \in \mathbb{Z}^+$ where $x_i$ 
		is the corresponding copy of $x$.
		
\item 	For all formulas $E$ satisfying $V_i(\dmd A, x_i) \le V_i(E,x_i) < V_i(B,x_i)$: 
		\[
		V_i(B, x_i) - 1/i < V_i(E,x_i) < V(B,x_i).
		\]
		
\end{enumerate}
Now we define a model $K' = \langle W',R',V' \rangle$ where:
\begin{enumerate}[(1)]
\item $W' = \{x_0\} \cup \bigcup_{i \in \mathbb{Z}^+} W_i$
\item	$R'yz = Ê\begin{cases}
				R_iyz							& y,z \in W_i\\
				R_ix_iz							& y = x_0, z \in W_i\\			
				0 								& z = x_0
				\end{cases}$
\item	$V'(p,y) = V_i(p,y)$ for all $y \in W_i$ and $V'(p,x_0) = 0$.
\end{enumerate}
But then:
\[
\begin{array}{rcl}
V'(\dmd A,x_0) 	& = & \sup_{y \in W}(\min(V(A,y),R'x_0y)\\
				& = & \sup \{\sup_{y \in W_i}(\min(V(A,y),R_ix_iy) \mid i \in \mathbb{Z}^+\}\\
				& = & \sup \{V_i(\dmd A,x_i) \mid i \in \mathbb{Z}^+\}\\
				& = & V'(B,x_0).
\end{array}\]
Now consider $(C \ltoe D) \in \sor$, recalling that $C$ and $D$ are diamond formulas, $\bot$, or $\top$. 
Clearly, if $\ltoe$ is $<$, then $V'(C,x_0) \ge V'(D,x_0)$.  If $\ltoe$ is $\le$, then using $(\star\star)$, 
it follows that $V'(C,x_0) > V'(D,x_0)$. So $\sor \h \dmd A < B$ is not $\lgc{L}$-valid as required. \qed

The next lemma is particular to the diamond fragment of $\lgc{GK^F}$ and is necessary for the extra step, 
not required in the case of $\lgc{GK}$, of 
removing relations of the form $\top \le \dmd A$.

\begin{lem} \label{lemltmodal}
If $(\sor \h \top \le \dmd A)$ is modal, saturated, and $\lgc{GK^F}$-valid, then 
$\sor$ is $\lgc{GK^F}$-valid.
\end{lem}
\proof
We argue by contraposition. Suppose that $\sor$ is not $\lgc{GK^F}$-valid. Then 
there is a Kripke model $K = \langle W,R,V \rangle$ for $\lgc{GK^F}$ and $x \in W$ such that 
$V(C,x) \rtoe V(D,x)$ for all $(C \ltoe D) \in \sor$. Fix a value $\lam<1$ such that 
whenever $V(\dmd B,x) < 1$ for some $\dmd B$ occurring in $\sor$, also $V(\dmd B,x) < \lam$. 
We define a Kripke model $K' = \langle W',R',V' \rangle$ for $\lgc{GK^F}$ where:
\begin{enumerate}[(1)]
\item $W' = \{x_0\} \cup W$
\item	$R'yz = Ê\begin{cases}
				Ryz				& y,z \in W\\
				\min(\lam,Rxz) 	&  y = x_0, z \in W\\			
				0 				& z = x_0
				\end{cases}$
\item	$V'(p,y) = V(p,y)$ for all $y \in W$ and $V'(p,x_0) = 0$.
\end{enumerate}
But then:
\[
\begin{array}{rcl}
V'(\dmd A,x_0) 	& = & \sup_{y \in W}(\min(V'(A,y),R'x_0y))\\
				& = & \sup_{y \in W}(\min(V(A,y),\min(\lam,Rxy)))\\
				& = & \begin{cases}
V(\dmd A,x)	& {\rm if} \ V(\dmd A,x) < 1\\
\lam			& {\rm otherwise.}
\end{cases}
				\end{array} 
				\]
So $(\sor \h \top \le \dmd A)$ is not  $\lgc{GK^F}$-valid as required.
\qed

Note that the following lemma includes relations of the form $\top \le \dmd A$ and holds even in the 
case of $\lgc{GK^F}$ where they are not needed.

\begin{lem} 
If $\mdl{L} \{\dmd A_i \le \dmd B_i\}_{i=1}^n \h \bot < \dmd C \h \top \le \dmd D$ for $\lgc{L} \in \{\lgc{GK},\lgc{GK^F}\}$, then 
\[
\mdl{L} \bigwedge_{j \in J} A_j \le \bigwedge_{j \in J} B_j \h  \bot <  C \h \top \le D
\]
for some $\emptyset \subset J \subseteq \{1,\ldots,n\}$.
\end{lem}
\proof
We argue by contraposition; i.e., suppose that:
\[
\not \mdl{L}  \bigwedge_{j \in J} A_j \le \bigwedge_{j \in J} B_j \h  \bot < C \h \top \le D
 \quad \textrm{ for all } \emptyset \subset J \subseteq \{1,\ldots,n\}.
\]
We obtain a model for each $i \in \{1,\ldots,n\}$ as follows. By assumption:
\[
\not \mdl{L} A_i \land \ldots \land A_n \le B_i \land \ldots \land B_n\h  \bot <  C \h \top \le D.
\]
So we have $K_i = \langle W_i,R_i,V_i \rangle$ and  $x_i \in W_i$ (with each $W_i$ distinct) such that:
\begin{enumerate}[(1)]
\item	$V_i(A_i \land \ldots \land A_n,x_i) > V_i(B_i \land \ldots \land B_n,x_i)$
\item	$V_i(C) = 0$ and $V_i(D) < 1$.
\end{enumerate}
Moreover, without loss of generality we can assume:
\[
V_i(B_i,x_i) \le V_i(B_k,x_i) \quad \textrm{ and so } \quad V_i(A_k,x_i) > V_i(B_i,x_i)  \quad \textrm{ for }k = i \ldots n.
\]
Now using Lemma~\ref{lemscaletwo} as in the case of Lemma~\ref{lemmodal}, we define iteratively $K'_i = \langle W_i,R'_i,V'_i \rangle$ for 
$i = n \ldots 1$ such that for $j = i \ldots n$:
\begin{enumerate}[(1)]
\item	$V'_k(B_j,x_k)  < V_k(A_j,x_j)$ for $k = 1 \ldots i-1$
\item	$V'_j(B_j,x_j)  < V'_k(A_j,x_k)$ for $k = i \ldots n$
\item	$V'_j(C,x_j) = 0$ and $V'_j(D,x_j) < 1$.
\end{enumerate}
Finally, we define a model $K = \langle W, R, V \rangle$ where for a new world $x_0$:
\begin{enumerate}[(i)]
\item		$W = W_1 \cup \ldots \cup W_n \cup \{x_0\}$
\item			$Rxy = Ê\begin{cases}
				R_ixy							& x,y \in W_i\\
				1								& x = x_0, y = x_i\ {\rm for}\ i \in \{1 \ldots n\}\\			
				0 								& y = x_0
				\end{cases}$

\item		$V(p,x) = V'_i(p,x)$ for all $x \in W_i$ and $V(p,x_0) = 0$.
\end{enumerate}
As in Lemma~\ref{lemmodal}, we obtain $\not \mdl{L} \{\dmd A_i \le \dmd B_i\}_{i=1}^n \h \bot < \dmd C \h \top \le \dmd D$   as required.   \qed

Our desired result is then obtained following  the same pattern as in the completeness proof 
for the box case.

\begin{thm}
For $\lgc{L} \in \{\lgc{GK},\lgc{GK^F}\}$: \ $\mdl{L} \sor$ iff $\der{SL_\dmd} \sor$. \qed
\end{thm}


\subsection{Consequences}

As in the case of the box fragment, we can use our completeness result for sequent of relations calculi 
to establish completeness also for the axiomatizations of the diamond fragments presented in Section~\ref{ss:boxdiamond}.

\begin{thm}
For $\lgc{L} \in \{\lgc{GK}, \lgc{GK^F}\}$: \ 
$\mdl{L} A$ iff $\der{HL_\dmd} A$.
\end{thm}
\begin{proof}
Let $\lgc{L} \in \{\lgc{GK}, \lgc{GK^F}\}$. 
Following the proof of Theorem~\ref{thmaxcomp}, it suffices to show that for each rule $\sor_1,\ldots,\sor_n \ / \ \sor$ of $\lgc{SL_\dmd}$, 
whenever $\der{\lgc{HL_\dmd}} I(\sor_i)$ for $i = 1 \ldots n$, also $\der{\lgc{HL_\dmd}} I(\sor)$. In particular, let us just consider 
the case of $\mdr$ for $\lgc{HGK_\dmd}$, since the case of $\mdr^*$ for $\lgc{HGK^F_\dmd}$ follows exactly the same pattern. 
 Using the derivabilities    $\der{\lgc{HGK_\dmd}} 
((A \to \dmd B_1) \lor (A \to \dmd B_2)) \leftrightarrow (A \to \dmd (B_1 \lor B_2))$ and 
$\der{\lgc{HGK_\dmd}} (\lnot \dmd A \land \lnot \dmd B) \leftrightarrow \lnot \dmd (A\lor B)$, it is enough to 
show that $\der{\lgc{HGK_\dmd}} \lnot C \to ((A \to B) \lor D)$ implies 
$\der{\lgc{HGK_\dmd}} \lnot \dmd C \to ((\dmd A \to \dmd B) \lor \dmd D)$. Suppose 
that $\der{\lgc{HGK_\dmd}} \lnot C \to ((A \to B) \lor D)$. Then using the derivability 
$\der{\lgc{HG}} (\lnot \lnot F \to G) \leftrightarrow (G \lor \lnot F)$, 
we have $\der{\lgc{HGK_\dmd}} \lnot \lnot C \lor ((A \to B) \lor D)$.
Applying $\nec_\dmd$, we get $\der{\lgc{HGK_\dmd}} (\dmd A \to \dmd B) \lor \dmd (\lnot \lnot C\lor D)$, 
and using $(K_\dmd)$, we have $\der{\lgc{HGK_\dmd}}  (\dmd A \to \dmd B) \lor (\dmd \lnot \lnot C \lor \dmd D)$.
Using $(Z_\dmd)$, we obtain $\der{\lgc{HGK_\dmd}}  (\dmd A \to \dmd B) \lor (\lnot \lnot\dmd  C \lor \dmd D)$, and finally  
$\der{\lgc{HGK_\dmd}} \lnot \dmd C \to ((\dmd A \to \dmd B) \lor \dmd D)$ as required. \end{proof}

The proof of the following complexity results follows exactly the same pattern as in the proof of 
Theorem~\ref{thm:boxpspace} for the box fragment of $\lgc{GK}$.

\begin{thm}
The validity problems for the diamond fragments of $\lgc{GK}$ and $\lgc{GK^F}$ are PSPACE-complete. \qed
\end{thm}

Note finally that it is not so easy to extend the hypersequent calculus $\lgc{GG}$ to calculi for the 
diamond fragments, since there is no natural way to interpret strict inequality relations of the form 
$\bot < A$ as sequents. One option would be to add decomposition rules for dealing with formulas 
$A \to \bot$ on the left of sequents occurring in a hypersequent. Completeness of the cut-free calculus could then 
be proved via a translation into sequents of relations. However, establishing cut elimination would be 
complicated and it is difficult to see how such a calculus could be easily extended to the first-order 
level or adapted to other logics.


\section{Discussion} \label{finalsection}

In this paper we have presented proof systems for the diamond and box fragments of two minimal normal modal logics based on 
a G{\"o}del fuzzy logic $\lgc{GK}$,  where the accessibility relation is classical (crisp), and $\lgc{GK^F}$,  where the accessibility 
relation is fuzzy. In particular, we have introduced a sequent of relations calculus and a hypersequent calculus (admitting 
cut-elimination) for the box fragment of $\lgc{GK}$, which coincides with the same fragment of $\lgc{GK^F}$, and sequent 
of relations calculi for the (distinct) diamond fragments of $\lgc{GK}$ and $\lgc{GK^F}$. We have used the calculi to establish 
completeness for corresponding axiomatizations of the fragments (a new result in the case of the diamond fragment of 
$\lgc{GK}$) and to establish new decidability and  PSPACE complexity bounds. Finally, in this section, we discuss some related  
work and avenues for further research.


\subsection{Related Proof Systems}

As observed already in the introduction, numerous proof systems for G\"odel logic may be found in the literature, 
encompassing  sequent calculi~\cite{son:gen,Dyckhoff99},  hypersequent calculi~\cite{Avron91b,AvrKon:00}, sequent of 
relations calculi~\cite{Baaz:1999:ACP}, decomposition systems~\cite{AvrKon:00}, graph-based methods~\cite{Lar07}, 
and goal-directed systems~\cite{met:book}. Our choice of the sequent of relations and hypersequent frameworks for 
 G\"odel modal logics was guided by a number of factors. First, like the sequent, decomposition, graph-based, and 
 goal-directed systems, and also the hypersequent calculus $\lgc{GLC^*}$ of~\cite{AvrKon:00}, 
 the sequent of relations calculus for G\"odel logic has invertible logical rules. Unlike most of 
 these other systems, however, the rules deal directly with the top-level connective of formulas and may be described 
 as reasonably ``natural'' or at least relatively easy to understand. Most significantly, perhaps, the framework facilitates a relatively 
 straightforward addition of modalities and easier completeness proofs than would be obtained in, e.g., a sequent calculus 
 framework. Moreover, it seems (without working through all the details) that the rules obtained and 
 completeness proofs for other calculi with invertible rules would be very similar and indeed could be obtained by 
 translating our results into these frameworks. For example, a multiple conclusion sequent $A_1,\ldots,A_n \seq B_1,\ldots,B_m$ 
 interpreted as $(A_1 \land \ldots \land A_n) \to (B_1 \lor \ldots \lor B_m)$ is $\lgc{G}$-valid iff the sequent of relations 
 $A_1 < \top \h \ldots \h A_m < \top \h \top \le B_1  \h \ldots \h \top \le B_m$ is $\lgc{G}$-valid. Conversely, a sequent of relations 
 $A_1 < B_1 \h \ldots \h A_n < B_n \h C_1 \le D_1 \h \ldots \h C_m \le D_m$ is $\lgc{G}$-valid iff the sequent 
 $B_1 \to A_1, \ldots, B_n \to A_n \seq C_1 \to D_1,\ldots, C_m \to D_m$ is $\lgc{G}$-valid.
 
 We do not of course mean to suggest that the other mentioned frameworks cannot be useful tools for investigating G\"odel modal 
 logics. Indeed, while our sequent of relations calculi provide ready-to-implement decision procedures for the box and diamond fragments 
 of $\lgc{GK}$ of optimal complexity and avoiding loop checking, it may be that other calculi provide a more suitable 
 basis for developing automated reasoning methods. In particular, there exist fast graph-based methods (see~\cite{Lar07}) for more 
 details) for the P-time problem  of  checking whether an atomic sequent of relations is $\lgc{G}$-valid (required in our algorithm before 
 applying the modal rules). Other gains in efficiency may be obtained by avoiding duplication of subformulas using graph-based 
 representations of sequents of relations, and by implementing some form of ``goal-directed'' proof search. We mention also 
 that our calculi provide a starting point for developing algorithms for fuzzy description logics based on G{\"o}del logic (as opposed 
 to the systems based on finite-valued or witnessed models~\cite{haj:making,Bobillo09}), although since our results are so far restricted 
 to the box and diamond fragments, this development would depend on the language under consideration.
 
The hypersequent calculus $\lgc{GG}$ does not have invertible logical rules and is useful, not as a basis for automated reasoning 
methods, but rather as a suitable tool for tackling theoretical problems for G\"odel logic. Unlike the other mentioned calculi, it 
has been extended to the first-order level (preserving cut-elimination) and used to establish properties such as Herbrand's theorem, 
Skolemization, and standard completeness for G\"odel logic~\cite{BaaZac00,met:book}. Similarly, the hypersequent calculus 
$\lgc{GGK_\mbx}$ for the box fragments of $\lgc{GK}$ and $\lgc{GK^F}$ can be extended to a first-order system admitting 
cut-elimination, although establishing completeness with respect to a semantics based on Kripke frames could be 
a challenging problem. Let us remark also that since hypersequents provide a general uniform framework for fuzzy and other 
substructural logics (see~\cite{met:book}), this may hold also for modal fuzzy logics. Certainly, structural rules can be removed 
from $\lgc{GGK_\mbx}$ to obtain hypersequent calculi admitting cut-elimination. The challenge then, as in the first-order case, is 
to relate such calculi to a semantics based on Kripke frames. A further limitation currently is also that we have a hypersequent 
calculus only for the box fragment. Calculi for the diamond fragments can be obtained by translation from the sequent of relations systems 
but uniformity and extensions to the first-order level or other logics are lost.


\subsection{Related Logics}

G\"odel logic is renowned not just as a fuzzy logic but also as a well-known intermediate logic. It therefore makes 
sense to ask what relationship $\lgc{GK}$ and $\lgc{GK^F}$ bear to modal intuitionistic and intermediate logics found in the literature. 
From a semantic perspective, the approaches are significantly different. Kripke models for the most popular 
intuitionistic modal logic $\lgc{IK}$ (see, e.g.,~\cite{Sim94}) and other modal intermediate logics (see, e.g.,~\cite{Wolter97}) 
make use of two accessibility relations, one for the modal operator and another for the intuitionistic connectives. 
Since Kripke models for G\"odel logic are linearly ordered, the resulting Kripke models for G\"odel modal logics developed in this way are also 
linear, which is not the case in general for Kripke models for $\lgc{GK}$ or $\lgc{GK^F}$. A closer comparison semantically could be made 
with an ``intuitionistic modal logic'' defined by considering standard modal Kripke frames equipped with an intuitionistic Kripke frame 
at each node.  Nevertheless, from a syntactic perspective, $\lgc{GK}$ and $\lgc{GK^F}$ may be viewed as  intermediate modal logics in the 
sense that each $\lgc{GK}$-valid formula (which includes each $\lgc{GK^F}$-valid formula)  is valid in $\lgc{K}$ and also every $\lgc{IK}$-valid formula is 
$\lgc{GK^F}$-valid  (and therefore also $\lgc{GK}$-valid). The latter follows easily from the fact that  (see, e.g.,~\cite{Sim94}) 
an axiomatization for $\lgc{IK}$ is given by an axiomatization for intuitionistic logic extended with $(K_\mbx)$, $(K_\dmd)$, $(F_\dmd)$, 
$\nec_\mbx$, and the $\lgc{GK^F}$-valid ``connecting axioms'' $\mbx (A \to B) \to (\dmd A \to \dmd B)$ and $(\dmd A \to \mbx B) \to \mbx (A \to B)$. 

We may also consider the situation regarding modal finite-valued G\"odel logics, obtained from our definitions by restricting 
valuations to $0$, $1$, and a particular finite number of truth values in $[0,1]$. We expect that sequent of relations calculi can be 
obtained for the same fragments of modal finite-valued G\"odel logics using the modal rules provided in this paper but adding axioms for 
the propositional case. However, developing these calculi is not so interesting from a theoretical perspective, since the logics 
are already known to be PSPACE-complete~\cite{Bobillo09} (proved in the context of G\"odel description logics by mapping to the classical case). 
Similarly, we can consider {\em witnessed} G\"odel modal logics where the values of box formulas $\mbx A$ and diamond formulas $\dmd A$ 
are attained by $A$ at some world. Calculi for these logics can (we expect) be developed in our framework, but it also seems possible to obtain 
PSPACE-completeness results using the methods of~\cite{Bobillo09} by providing a bound (based on the formula to be proved) to reduce validity to 
validity in a modal finite-valued G\"odel logic.


\subsection{Further Work}

As noted in the introduction, this paper provides only a starting point for investigating G\"odel modal logics. 
The most pressing concern is to extend our proof-theoretic treatment to the full language of $\lgc{GK}$ and $\lgc{GK^F}$, 
and thereby obtain axiomatization, decidability, and complexity results for these logics.  It has been conjectured by Caicedo and 
Rodr{\'i}guez that an axiomatization for $\lgc{GK^F}$ is provided by the axioms and rules for the box and diamond 
fragments extended with the connection axioms of $\lgc{IK}$ (see the previous subsection), but a proof using their algebraic 
approach is still lacking. From our more proof-theoretic perspective, the difficulty is that appropriate sequent of relations rules 
should deal with several combinations of box and diamond fragments occurring in relations, i.e., $\mbx A \ltoe \mbx B$, 
$\mbx A \ltoe \dmd B$, $\dmd A \ltoe \mbx B$, $\dmd A \ltoe \dmd B$ where $\ltoe$ is $\le$ or $<$, plus relations involving 
$\bot$ or $\top$. The combinatorial nature of the required rules leads to an explosion in the number of cases that should be 
considered in the completeness proof, and it is not clear that the ideas developed in this paper of squeezing or pushing up 
or down valuations suffice to cope with all possibilities. It may therefore be preferable to consider a more general framework 
than sequent of relations,  perhaps using labels to encode semantic information.

We also intend to consider stronger modal logics such as the box and diamond fragments axiomatized in \cite{CR08} 
based on  fuzzy Kripke frames satisfying further properties of reflexivity, transitivity, seriality, and symmetry. Our conjecture is that we 
can extend the sequent of relation calculus to handle all these conditions except possibly symmetry. Again, however, 
a more general (labelled) framework might be useful. Finally, a further direction of research is the extension of our proof systems 
to a richer language comprising  multiple modalities  and truth constants, interesting for applications to  fuzzy description logics 
(see, e.g.,~\cite{straccia01a,haj:making,Bobillo09}). 


\bibliographystyle{plain}

\end{document}